\documentclass{article} 
\usepackage{amssymb}
\usepackage[normalem]{ulem}
\usepackage{tikz}
\usepackage{mathtools}

\def\lc{\mbox{\rm lc}}
\newcommand{\ord}{\mathrm{ord}\,}
\newcommand{\Res}{\mathrm{Res}\,}
\newcommand{\cont}{\mathrm{cont}}
\newcommand{\bC}{\mathbf{C}}
\newcommand{\bR}{\mathbf{R}}
\newcommand{\bK}{\mathbf{K}}
\newcommand{\bN}{\mathbf{N}}

\newcommand{\bQ}{\mathbf{Q}}
\newcommand{\bL}{\mathbf{L}}
\newcommand{\bM}{\mathbf{M}}
\newcommand{\Zer}{\mathop\mathrm{Zer}}
\newcommand{\scalar}[2]{\langle #1,#2\rangle}

\newtheorem{Theorem}{Theorem}[section]
\newtheorem{Proposition}[Theorem]{Proposition}
\newtheorem{Corollary}[Theorem]{Corollary}
\newtheorem{Remark}[Theorem]{Remark}

\newtheorem{Example}[Theorem]{Example}
\newtheorem{Lemma}[Theorem]{Lemma}
\newtheorem{Definition}[Theorem]{Definition}

\newenvironment{proof}[1][Proof]{\textbf{#1.} }{\
\rule{0.5em}{0.5em}}

\newcommand{\Teissr}[4]{
   \setlength{\unitlength}{1ex}
   \begin{picture}(#3,3)(0,0.4)
      \put(0,1.15){\line(1,0){#3}}
      \put(0,0.85){\line(1,0){#3}}
      \put(#4,1.3){\makebox(0,0)[b]{$#1$}}
      \put(#4,0.7){\makebox(0,0)[t]{$#2$}}
   \end{picture}}

\begin{document}
\title{Higher order polars of quasi-ordinary singularities
\footnotetext{
      \begin{minipage}[t]{4in}{\small
       2000 {\it Mathematics Subject Classification:\/} Primary 32S05;
       Secondary 14H20.\\
       Key words and phrases: quasi-ordinary polynomial, higher order polar, factorization, P-contact, self-contact.\\
       The first-named author was partially supported by the Spanish Project
       MTM 2016-80659-P.}
       \end{minipage}}}
 \author{Evelia R.\ Garc\'{\i}a Barroso and Janusz Gwo\'zdziewicz}
\maketitle

\begin{abstract}
A quasi-ordinary polynomial is a monic polynomial with coefficients in the power series ring such that its discriminant equals a monomial up to unit. In this paper we study higher derivatives of quasi-ordinary poly\-nomials, also called higher order polars. We find factorizations of these polars. Our research in this paper goes in two directions. We generali\-ze the results of  Casas-Alvero  and our previous results on higher order polars in the plane to irreducible quasi-ordinary polynomials. We also generalize the factorization of the first polar of a  quasi-ordinary polynomial (not necessary irreducible) given by the first-named author and Gonz\'alez-P\'erez  to higher order polars.   
This is a new result even in the plane case. Our results remain true when we replace quasi-ordinary polynomials by quasi-ordinary power series.
\end{abstract}

\section{Introduction}
\label{intro}
In \cite{Merle} Merle gave a decomposition theorem of a generic polar curve of an irreducible plane curve singularity, according to its topological type. The factors of this decomposition are not necessary irreducible. Merle`s decomposition was generalized to  reduced plane curve germs by Kuo and Lu \cite{K-L}, Delgado de la Mata \cite{Delgado}, Eggers \cite{Eggers}, Garc\'{\i}a Barroso \cite{GB} among others. In \cite{GB-GP}, Garc\'{\i}a Barroso and Gonz\'alez P\'erez obtained decompositions of the polar hypersurfaces of quasi-ordinary singularities. On the other hand, Casas-Alvero in \cite{Casas} generalized the results of Merle to higher order polars of an irreducible plane curve. In \cite{Forum} we improved his results  giving a finer decomposition  in such a way that we are able to determine the topological type of some irreducible factors of the polar as well as their number.

\medskip

\noindent Our research in this paper goes in two directions. We generalize the results of  \cite{Casas}  and \cite{Forum} on higher order polars to irreducible quasi-ordinary singularities (see Theorem \ref{Merle} and Proposition \ref{ppppp}). We also generalize the factorization of the first polar of a  quasi-ordinary singularity (not necessary irreducible) from  \cite{GB-GP}  to higher order polars (see Theorem \ref{dec-red-qo}).   
This is a new result even in the plane case.

\medskip

\noindent Our approach is based on  Kuo-Lu trees, Eggers trees, Newton polytopes and resultants. As it was remarked in \cite{tesis} and \cite{GB-GP}, the irreducible factors of  the polar of a quasi-ordinary singularity are not necessary quasi-ordinary. For that reason, we mesure the relative position of these irreducible factors and those of the quasi-ordinary singularity  using a new notion called the {\em P-contact}, which plays in our situation the role of the {\em logarithmic distance} introduced by P\l oski in \cite{Ploski}. 

\medskip

\noindent The paper is organized as follows. 

\medskip

\noindent In Section \ref{section-Newton-polytopes} we recall the notion of  the Newton polytope of  a Weierstrass polynomial $f\in \bK[[\underline{x}]][y]$ and we use it together with the Rond-Schober irreducibility criterium \cite{R-S}, in order to give sufficient conditions for the reducibility of $f$. The most important result in this section is Corollary \ref{irred}, which allows us to characterize, in Theorem \ref{pack}, the irreducible factors of the higher order polars of the polynomial $f$.

\medskip

\noindent In Section \ref{section-Kuo-Lu-tree} we present the notion of the Kuo-Lu tree of a quasi-ordinary Weierstrass polynomial.  Then in Section \ref{section-Compatibility with pseudo-balls} we identify the bars of a Kuo-Lu tree with certain sets of fractional power series called {\em pseudo-balls} and we  introduce the notion of {\em compatibility} of a  Weierstrass polynomial with a pseudo-ball. Every quasi-ordinary Weierstrass polynomial is compatible with every pseudo-ball associated with its Kuo-Lu tree. Moreover if a Weierstrass  polynomial is compatible with a pseudo-ball then any factor of it is  compatible too (see Corollary \ref{CCKiel}). In Lemma \ref{derivatives} we prove that, under some conditions, the normalized higher derivatives inherit the compatibility property.   In Section \ref{section-Conjugate-pseudo-balls} we introduce, using Galois automorphisms,  an equivalence relation in the set of pseudo-balls, called {\em conjugacy},  and we explore the compatibility property for conjugate pseudo-balls. We generalize the Kuo-Lu Lemma \cite[Lemma 3.3]{K-L} to higher derivatives in Section \ref{section-Kuo-Lu Lemma}. In Section \ref{section-Newton-polytopes-resultants} we introduce our main tool, monomial substitutions, that allows us to reduce several questions to the case of two variables. In particular, if   $f$ and $g$ are power series  in $d+1$ variables such that after generic monomials substitutions we obtain power series $\bar f,\bar g$ in two variables with equal Newton polygons,  then the Newton polytopes of $f$ and $g$ are also equal (see Corollary \ref{R1}).
In Section \ref{section-Eggers-tree} we extend the notion of Eggers tree introduced in  \cite{Eggers}, to quasi-ordinary settings. Remark that the tree we use here  is not exactly the Eggers-Wall tree introduced in \cite{tesis} for the quasi-ordinary situation. The main
 result of Section \ref{section-Irreducible factors} is Theorem \ref{pack}, where we characterize the irreducible factors of higher derivatives of  quasi-ordinary Weierstrass polynomials.  Theorem \ref{pack} allows us to give factorizations of higher derivatives,  in terms of the Eggers tree,  in Section \ref{section-Eggers factorization}. Theorem \ref{dec-red-qo} generalizes the factorization from \cite{Casas} on higher order polars to quasi-ordinary singularities (not necessary irreducible) and also the factorization from \cite{GB-GP} to higher order polars. Theorem \ref{Merle} and Proposition \ref{ppppp} extend the statements of \cite[Theorem 6.2]{Forum} to irreducible quasi-ordinary Weierstrass polynomials. Finally in Section \ref{section-Eggers series} we establish that  our results also hold for quasi-ordinary power series.

\section{Newton polytopes}
\label{section-Newton-polytopes}
\noindent 
Let $\alpha=\sum \alpha_{\bf i} \underline{x}^{\bf i}\in S[[\underline{x}]]$
be a non zero  formal power series with coefficients in a ring $S$, 
where $\underline{x}=(x_1,\ldots, x_d)$ 
and $\underline{x}^{\bf i}=x_1^{i_1}\cdots x_d^{i_d}$, with ${\bf i}=(i_1,\ldots,i_d)$. 
The {\em Newton polytope} $\Delta(\alpha)\subset \bR^d$ of $\alpha$ 
is the convex hull of the set $\bigcup_{\alpha_{\bf i}\neq 0} {\bf i}+\bR^d_{\geq 0}$. 
By convention the Newton polytope of the zero power series is the empty set.

\noindent  The Newton polytope of a polynomial $f=\sum_{{\bf i},j}a_{{\bf i},j}\underline{x}^{\bf i}y^j\in S[[\underline{x}]][y]$ is the polytope $\Delta(f)\subset \bR^d\times \bR$ 
of $f$ viewed as a power series in $x_1,\ldots, x_d,y$. 
If $\Gamma$  is a compact face of $\Delta(f)$ then $f|_{\Gamma}:=\sum_{({\bf i},j)\in {\Gamma}}a_{{\bf i},j}\underline{x}^{\bf i}y^j\in S[\underline{x}][y]$ is called  the {\em symbolic restriction} of $f$ to $\Gamma$. \\

\smallskip

\noindent We say that a subset of  $\bR^{d+1}$ is a {\em Newton polytope} if it is the Newton polytope of  some polynomial in $S[[\underline{x}]][y]$.

\medskip
\noindent 
Let ${\bf q}=(q_1,\dots,q_d)\in\bQ_{\geq0}^d$ and let 
$k$ be a positive integer. We define the {\em elementary Newton polytope} 
\[
\Bigl\{\Teissr{{\bf q}}{k}{3}{1.5}\Bigr\}:=
\mbox{convex hull}\;\bigl( \{\, (q_1,\dots,q_d,0), (0,\dots,0,k)\,\}
+\bR_{\geq0}^{d+1}\bigr)\;.
\]

\noindent 
Its {\em inclination} is, by definiton, $\frac{1}{k}{\bf q}$.

\noindent 
We denote by $\Bigl\{\Teissr{\infty}{k}{3}{1.5}\Bigr\}$ the Newton polytope $\Delta(y^k)$, which is the first orthant translated by $(0,\dots,0,k)$. By convention we consider it as an elementary polytope.

\begin{Example} The elementary Newton polytope $\left\{\Teissr{(4,2)}{8}{5}{2.5}\right\}$ is
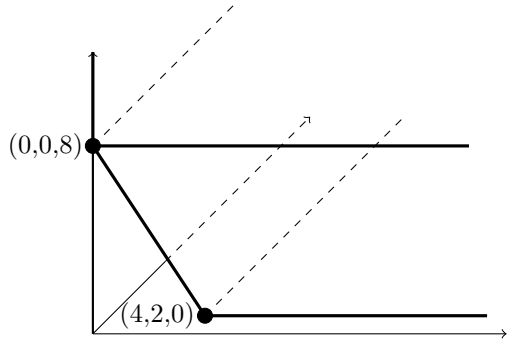
\begin{center}
\begin{tikzpicture}[scale=2.5]
\draw [->](0,0,0) -- (2.2,0,0); \draw[->](0,0,0) -- (0,1.5,0) ; \draw(0,0,0) -- (0,0,-1);
\draw[thick] (0,0,0) -- (0,1,0) node[left] {$(0,\!0,\!8)$};
\draw[dashed](0,1,0) -- (0,1,-2);
\draw[very thick](0,1,0)-- (2,1,0);
\draw[very thick](0,1,0)-- (0,1.5,0);
\draw[very thick] (0,1,0) -- (0.5,0,-0.25) node[left] {$(4,\!2,\!0)$};
\draw[very thick](0.5,0,-0.25) -- (2,0,-0.25);
\draw[dashed](0.5,0,-0.25) -- (0.5,0,-3);
\draw[->][dashed](0,0,-1) -- (0,0,-3);
\node[draw,circle,inner sep=2pt,fill=black] at (0,1,0) {};
\node[draw,circle,inner sep=2pt,fill=black] at (0.5,0,-0.25) {};
\end{tikzpicture}
\end{center}
\end{Example}

\medskip

\noindent A Newton polytope is {\em polygonal} if the maximal dimension of its compact faces is one. 

\medskip

\noindent Remember that the Minkowski sum of $A,B\subset  \bR^{d+1}$ is the set $A+B:=\{a+b\;:\;a\in A,b\in B\}$.  
If a Newton polytope $\Delta$
has a representation of the type
\begin{equation}
\label{canonical}
 \Delta= \sum_{i=1}^r \left\{\Teissr{{\bf q}_i}{k_i}{2}{1}\right\},
\end{equation}

\noindent then summing all the elementary Newton polytopes of  the same inclination in \eqref{canonical} we obtain a unique representation, up to the order of the terms,  called {\em canonical representation} of $\Delta$. If the inclinations can be well-ordered then $\Delta$ is polygonal.
\bigskip

\subsection{Newton polytopes and  factorizations}

\noindent Let $\bK$ be a field of characteristic zero.
We denote by $\bK[[x_1^{1/k},\dots, x_d^{1/k}]]$ the ring 
of fractional power series in $d$ variables where all the exponents 
are nonnegative rational numbers with denominator $k\in \bN\backslash\{0\}$.
Put $\bK[[\underline{x}^{1/\bN}]]:=
\bigcup_{k\in \bN\backslash\{0\}}\bK[[x_1^{1/k},\dots, x_d^{1/k}]]$. We will denote by 
\[
\alpha\bK[[\underline{x}^{1/\bN}]]=\{\alpha w: w\in \bK[[\underline{x}^{1/\bN}]]\}
\]

\noindent the ideal of $\bK[[\underline{x}^{1/\bN}]]$ generated by $\alpha \in \bK[[\underline{x}^{1/\bN}]]$. 

\medskip
\noindent A {\em Weierstrass polynomial} is a monic polynomial where the coefficients different from the leading coefficient are non-units of the formal power series. Notice that, according to this definition, the constant polynomial 1 is a Weierstrass polynomial.

\medskip
\noindent The next lemma gives sufficient conditions for reducibility of Weierstrass polynomials. One of the consequences of this lemma is that a Weierstrass polynomial 
with a polygonal Newton polytope admits a decomposition 
into coprime factors such that the Newton polytope of each factor is elementary (see Theorem \ref{Th:decomp}, see also \cite[Theorem 3]{GB-GP}). 

\begin{Lemma}
\label{RS}
 Let $g=y^m+c_1 y^{m-1}+\cdots + c_m\in \bK[[\underline{x}]][y]$ be a Weierstrass polynomial. Assume that there exists ${\bf q} \in \bQ^d$ such that 
$c_i\bK[[\underline{x}^{1/\bN}]]\subseteq \underline{x}^{i {\bf q}}\bK[[\underline{x}^{1/\bN}]]$ for all $1\leq i\leq m$ with equality for some $i=i_0$, $1\leq i_0 <m$ and strict inclusion for $i=m$. Then $g$ has at least two coprime factors.
\end{Lemma}
\begin{proof} We will apply \cite[Theorem 2.4]{R-S}. Without lost of generality we may assume that $i_0$ is the maximal index $i\in\{1,\ldots,m-1\}$ such that $c_i\bK[[\underline{x}^{1/\bN}]]=\underline{x}^{i {\bf q}}\bK[[\underline{x}^{1/\bN}]]$. Then the segment $\Gamma$ with endpoints $(0,\dots,0, m)$ and $(i_0{\bf q}, m-i_0)$ is an edge of~$\Delta(g)$. The symbolic restriction of $g$ to $\Gamma$ is the product $g|_{\Gamma}=y^{m-i_0}\cdot \tilde g$, where $\tilde g \in \bK[\underline{x}][y]$  is coprime with $y$. The associated polyhedron of $g$, in the sense of Rond-Schober (see \cite[page 4732]{R-S} is $m{\bf q}+ \bR_{\geq0}^{d}$. Hence the polynomial $g$ verifies the hypothesis of  \cite[Theorem 2.4]{R-S} and the lemma follows. 
\end{proof}

\medskip 
\noindent 
\begin{Remark}
\label{geomRS}
The assumptions of Lemma~\ref{RS} mean geometrically that the Newton polytope $\Delta(g)$
is included in the elementary polytope $\left\{\Teissr{m{\bf q}}{m}{4}{2}\right\}$, and $\Delta(g)$ has an edge $\Gamma$, which endpoints $(0,\dots,0, m)$ and $(i_0{\bf q}, m-i_0)$, for some $1\leq i_0<m$. The next picture illustrates the situation:

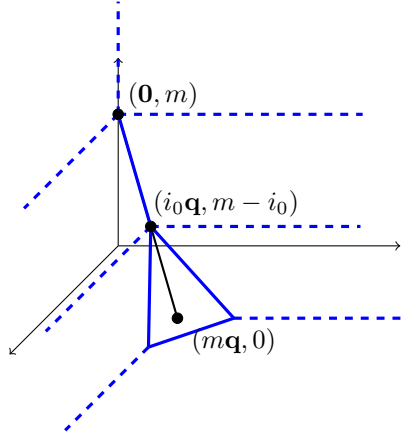
\begin{center}
\begin{tikzpicture}[scale=2.5]
\draw [->](0,0,0) -- (1.5,0,0); \draw[->](0,0,0) -- (0,1,0) ; 
\draw[->](0,0,0) -- (0,0,1.5);
\draw[thick] (0,0.7,0) -- (0.7,0,1);
\draw[very thick, dashed, color=blue] (0,0.7,0) -- (0,1.3,0);
\draw[very thick, dashed, color=blue] (0,0.7,0) -- (0,0.7,1.3);
\draw[very thick, dashed, color=blue] (0,0.7,0) -- (1.3,0.7,0);
\draw[very thick, color=blue] (0.385,0.315,0.55) -- (0,0.7,0);
\draw[very thick, color=blue] (0.385,0.315,0.55) -- (1,0,1);
\draw[very thick, color=blue] (0.385,0.315,0.55) -- (0.7,0,1.4);
\draw[very thick, color=blue] (1,0,1) -- (0.7,0,1.4);
\draw[very thick, dashed, color=blue ] (0.7,0,1.4) -- (0.7,0,2.55);
\draw[very thick, dashed, color=blue] (1,0,1) -- (1.9,0,1);
\draw[very thick, dashed, color=blue] (0.385,0.315,0.55) -- (1.5,0.315,0.55);
\draw[very thick, dashed, color=blue] (0.385,0.315,0.55) -- (0.385,0.315,2);
\draw (0,0.8,0) node[right]{$({\bf 0}, m)$};
\node[right,above] at (0.785,0.315,0.55) {$(i_0{\bf q}, m-i_0)$};
\draw (1,0,1) node[below]{$(m{\bf q},0)$};
\node[draw,circle,inner sep=1.4pt,fill=black] at (0,0.7,0) {};
\node[draw,circle,inner sep=1.4pt,fill=black] at (0.385,0.315,0.55) {};
\node[draw,circle,inner sep=1.4pt,fill=black] at (0.7,0,1) {};
\end{tikzpicture}
\end{center}
\end{Remark}

\begin{Theorem}\label{Th:decomp}
Let $f\in \bK[[\underline{x}]][y]$ 
be a Weierstrass polynomial.
Assume that $\Delta(f)$ is a polygonal Newton polytope  with canonical representation
$\sum_{i=1}^r \left\{\Teissr{{\bf q}_i}{k_i}{2}{1}\right\}$. 
Then $f$ admits a factorization $f_1\cdots f_r$, where 
$f_i\in \bK[[\underline{x}]][y]$ are Weierstrass polynomials, not necessarily irreducible, such that 
$\Delta(f_i)=\left\{\Teissr{{\bf q}_i}{k_i}{2}{1}\right\}$ for $i=1,\dots,r$.
\end{Theorem}

\noindent 
\begin{proof}
Let $f=g_1\cdots g_s$ be the factorization of $f$ into irreducible Weierstrass polynomials. 
Since  the Newton polytope of a product is the Minkowski sum of the Newton polytopes 
of the factors, by hypothesis
we get $\Delta(g_j)=\sum_{i=1}^r b_{ij} \left\{\Teissr{{\bf q}_i}{k_i}{2}{1}\right\}$ 
for some $b_{ij}\in \bQ_{\geq 0}$.
By Remark \ref{geomRS} $\Delta(g_j)$ is elementary, hence for fixed $j$  only one term of the previous sum is nonzero. On the other hand, for fixed $i$, we get $\sum_{j}b_{ij}=1$. Put $f_i:=\prod g_j$, where the product runs over all $g_j$ such that $b_{ij}\neq 0$. 
Then  $f=f_1\cdots f_r,$ where $\Delta(f_i)=\left\{\Teissr{{\bf q}_i}{k_i}{2}{1}\right\}$ 
for $i=1,\dots,r$.
\end{proof}

\begin{Theorem}\label{Th:discriminant}
Let $f(y),g(y)\in \bL[y]$ be monic polynomials, where $\bL$ is a field of characteristic zero. 
If $g(y)$ is irreducible in the ring $\bL[y]$ then the polynomial 
$R(T)=\Res_y(T-f(y),g(y))$, where $\Res_y(-,-)$ denotes the resultant, is either irreducible in $\bL[T]$ 
or is a power of an irreducible polynomial. 
\end{Theorem}

\noindent \begin{proof}
Let $y_1,\dots, y_m$ be the roots of $g(y)$ in the algebraic closure of 
the field~$\bL$. Then $R(T)=\prod_{i=1}^m(T-f(y_i))$. Since $\bL$ is a field of characteristic zero and $g(y)$ is irreducible, the Galois 
group of the field extension $\bL\hookrightarrow \bL(y_1,\dots, y_m)$ acts transitively on the set 
$\{y_1,\dots, y_m\}$. It follows that this group acts transitively on the set 
$\{f(y_1),\dots, f(y_m)\}$. Hence if $R=R_1\cdots R_s$ is a factorization of $R=R(T)$  into  irreducible monic polynomials in the ring $\bL[T]$ then $R_i=R_j$ for $i\neq j$. 
\end{proof}

\noindent Next corollary will be used  in the proof of  the main result of the decompositions of higher polars, which is Theorem \ref{pack}.
\begin{Corollary}\label{irred}
Let $f(y)$, $g(y) \in \bK[[\underline{x}]][y]$ be Weierstrass polynomials. 
If the resultant $\Res_y(g(y),f(y)-T) \in \bK[[\underline{x}]][T]$ satisfies the assumptions of Lemma~\ref{RS}, 
then $g(y)$ is not irreducible in the ring $\bK[[\underline{x}]][y]$.
\end{Corollary}

\noindent \begin{proof}
By Lemma \ref{RS} the polynomial $R(T)$ has at least two coprime factors. 
By Theorem \ref{Th:discriminant}, $g(y)$, considered as a polynomial 
in $\bK((\underline{x}))[y]$, is not irreducible, 
thus by Gauss Lemma it is not irreducible as a polynomial in $\bK[[\underline{x}]][y]$.
\end{proof}

\begin{Remark}
\label{Beata}
Beata Hejmej in \cite{Beata} generalizes Theorem \ref{Th:discriminant} to polynomials with coefficients in a field of any characteristic. Hence the results of this section hold for fields of arbitrary characteristic.
\end{Remark}

\section{Kuo-Lu tree of a quasi-ordinary polynomial}
\label{section-Kuo-Lu-tree}
\noindent 
From now  on $\bK$ will be an algebraically closed field of characteristic zero. Let  $f(y)\in \bK[[\underline{x}]][y]$ be a  Weierstrass polynomial of degree $n$. Such a polynomial is \emph{quasi-ordinary} if its $y$-discriminant 
equals $\underline{x}^{\bf i}u(\underline{x})$, where $u(\underline{x})$ is a unit in  $\bK[[\underline{x}]]$ and ${\mathbf i}\in \bN^{d}$. 
After Jung-Abhyankar theorem (see \cite[Theorem 1.3]{Parusinski-Rond}) the roots of $f$ are in the ring 
$\bK[[\underline{x}^{1/\bN}]]$
of fractional power series and we may factorize 
$f(y)$ as $\prod_{i=1}^n(y-\alpha_i)$, 
where $\alpha_i$ is zero or a fractional power series of nonnegative order.  
Put $\Zer f:=\{\alpha_i\,:\;1\leq i\leq n\}$. 
Since the differences of roots divide the discriminant,  for $i\neq j$ we have
\begin{equation}
\label{contact}
\alpha_i-\alpha_j=\underline{x}^{{\bf q}_{ij}}v_{ij}(\underline{x}),\;\;\;
\hbox{\rm for some  }{\bf q}_{ij} \in \bQ^d \;\hbox{\rm and  } v_{ij}(0)\neq0.
\end{equation}

\medskip
\noindent 
The {\em contact} of $\alpha_i$ and $\alpha_j$ is by definition $O(\alpha_i,\alpha_j):={\bf q}_{ij}$. By convention $O(\alpha_i,\alpha_i)=+\infty$. 

\medskip
\noindent 
We introduce  in $\bQ_{\geq 0}^d$ the partial order: ${\bf q} \leq {\bf q}'$ if ${\bf q}'-{\bf q} \in \bQ_{\geq 0}^d$. 
By convention $+\infty$ is bigger than any element of $\bQ_{\geq 0}^d$.

\medskip
\noindent 
After \cite[Lemma 4.7]{B-M}, for every $\alpha_i,\alpha_j,\alpha_k\in \Zer f$ one has $O(\alpha_i,\alpha_k)\leq O(\alpha_j,\alpha_k)$ or $O(\alpha_i,\alpha_k)\geq O(\alpha_j,\alpha_k)$.

\medskip
\noindent 
Moreover, we have the {\em strong triangle inequality:} 

\begin{equation}
\label{STI}
O(\alpha_i,\alpha_j)\geq \min \{O(\alpha_i,\alpha_k), O(\alpha_j,\alpha_k)\}.  \tag{STI}
\end{equation}

\medskip
\noindent In general, we say that the {\em contact} between  the fractional power series $\alpha$ and $\beta$ is {\em well-defined} if and only if $\alpha-\beta=\underline{x}^{{\bf q}}w(\underline{x})$, for some ${\bf q} \in \bQ^d$ and $w\in \bK[[\underline{x}^{1/\bN}]]$  such that  $w(0)\neq0$. In such a case we put $O(\alpha,\beta)={\bf q}$. 

\medskip
\noindent 
Now we construct the {\em Kuo-Lu tree} of a quasi-ordinary Weierstrass polynomial~$f$.
Given ${\bf q} \in \bQ^d_{\geq 0}$ we put  $\alpha_i \equiv \alpha_j$ mod ${\bf q}^{+}$ if $O(\alpha_i,\alpha_j)>{\bf q},$ for $\alpha_i,\alpha_j\in \Zer f$.
\noindent   Let ${\bf h}_0$ be the minimal contact between the elements of $\Zer f$. We represent $\Zer f$ as a horizontal bar $B_0$ and call $h(B_0)$ the {\em height} of $B_0$. The equivalence relation $\equiv$ mod $h(B_0)^+$ divides $B_0=\Zer f$ into cosets $B_1, \ldots, B_r$.  We draw $r$ vertical segments from the bar $B_0$ and at the end of the $j$th vertical segment we draw a horizontal bar which represents $B_j$. The bar $B_j$ is called a {\em postbar} of $B_0$ and in such a situation we write $B_0 \perp B_j$. We repeat this construction recursively for every $B_j$ with at least two elements.  The set of bars ordered by the inclusion relation is a tree. Following \cite{K-L} we call this tree the {\em Kuo-Lu tree} of $f$ and denote it $T(f)$. The bar $B_0$ of minimal height is called the {\em root} of $T(f)$. For every bar $B$ of $T(f)$ there exists a unique sequence $B_0\perp B' \perp B'' \perp \cdots \perp B$, starting in $B_0$ and ending in $B$. 

\medskip
\noindent In the above construction, we do not draw the bars $\{\alpha_i\}\subset \Zer f$. These bars are the {\em leaves} of $T(f)$ and they are the only bars of infinite height.

\medskip
\noindent Let $B,B'\in T(f)$ be such that $B\perp B'$. All fractional power series belonging to $B'$ have the same term with the exponent $h(B)$. Let $c$ be the coefficient of such term. We say that $B'$ {\em is supported at the point} $c$ on $B$ and we denote it by $B\perp_c B'$. Observe that different postbars of $B$ are supported at different points.

\medskip
\noindent This construction is adapted from \cite{K-L} to quasi-ordinary case. 

\begin{Example}\label{ex:KL}
Let $f=f_1f_2\in \bC[[x_{1},x_{2}]][y]$, where $f_1=y^2-x_1^{3}x_2^2$ and $f_2=y-x_1^{5}x_2^2$. Observe that $f$ is quasi-ordinary since its $y$-discriminant equals $4x_{1}^{9}x_{2}^{6}(-1+x_{1}^{7}x_{2}^{2})^{2}$. The roots of $f$ are  $\alpha=x_1^{3/2}x_2$, $\beta=-x_1^{3/2}x_2$ and $\gamma=x_1^{5}x_2^2$. The Kuo-Lu tree of $f$ is:
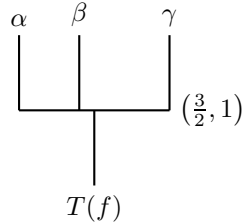
\begin{center}
\begin{tikzpicture}[scale=1]
\draw [-, thick](0,0) -- (0,1); 
\draw[-, thick](-1,1) -- (1,1); 
\draw[-, thick](-1,1) -- (-1,2); 
\draw[-, thick](1,1) -- (1,2); 
\draw[-, thick](-0.2,1) -- (-0.2,2); 
\node[right] at (1,1) {$\left(\frac{3}{2},1\right)$};
\node[above] at (-1,2) {$\alpha$};
\node[above] at (-0.2,2) {$\beta$};
\node[above] at (1,2) {$\gamma$};
\node[below] at (0,0) {$T(f)$};
\end{tikzpicture}
\end{center}
\end{Example}

\noindent In the above picture we draw also a vertical segment supporting $T(f)$ called by Kuo and Lu in \cite{K-L} the  {\em main trunk} of the tree.

\section{Compatibility with pseudo-balls}
\label{section-Compatibility with pseudo-balls}
Let $\alpha\in\bK[[\underline{x}^{1/\bN}]]$ be a fractional power series and ${\bf h}\in \bQ^d_{\geq 0}$. 
The {\em pseudo-ball} centered in $\alpha$ and of height ${\bf h}$ is the set $\alpha+\underline{x}^{\bf h}\bK[[\underline{x}^{1/\bN}]]$. The {\em pseudo-ball centered in $\alpha$ of infinite height} is the set
$\{\alpha\}$.

\medskip
\noindent
Let $f$ be a quasi-ordinary polynomial $f$. Consider the bar $B=\{\alpha_{i_1},\dots,\alpha_{i_s}  \}$ with finite height 
${\bf h}$ of the Kuo-Lu tree $T(f)$. Set 
$\tilde B:=\alpha + \underline{x}^{\bf h}\bK[[\underline{x}^{1/\bN}]]$, where 
$\alpha \in B$.
As $\alpha_{i_k}-\alpha_{i_l}\in \underline{x}^{\bf h}\bK[[\underline{x}^{1/\bN}]]$ for $1\leq k\leq l \leq s$  the pseudo-ball $\tilde B$ is independent of the choice 
of $\alpha$. 
If $B=\{\alpha_i \}$ is a bar of infinite height then we put $\tilde B= B$.
The mapping $B\to \tilde B$ is a one-to-one correspondence between $T(f)$ and the 
set of pseudo-balls 
$\tilde T(f):=\{\alpha_i+(\alpha_i-\alpha_j)\bK[[\underline{x}^{1/\bN}]]: \alpha_i, \alpha_j\in\Zer f\}$.
For the purposes of this article it is easier to deal with pseudo-balls, 
hence from now on, we shall identify the elements of $T(f)$ 
with corresponding pseudo-balls. 
Such pseudo-balls  will be called {\em quasi-ordinary pseudo-balls}. 

\medskip
\noindent 
Let $B=\alpha + \underline{x}^{h(B)}\bK[[\underline{x}^{1/\bN}]]$ be a quasi-ordinary pseudo-ball of finite height. 
Every $\gamma\in B$  has a form
$\gamma=\lambda_B(\underline{x})+c_{\gamma}\underline{x}^{h(B)}+\cdots$, 
where $\lambda_B(\underline{x})$ is obtained from any $\beta\in B$ 
by omitting all the terms of order bigger than or equal to $h(B)$ and ellipsis means terms of higher order.
We call the number $c_{\gamma}$ the {\em leading coefficient of } 
$\gamma$ {\em with respect to} $B$ and denote it $\lc_B(\gamma)$. 
Remark that $c_{\gamma}$ can be zero.

\medskip
\noindent
Let $\bL$ be the field of fractions of $\bK[[\underline{x}]]$. 
It follows from \cite[Remark 2.3]{GP} that any truncation of a root of a quasi-ordinary polynomial is a root of a quasi-ordinary polynomial. Hence the field extensions  $\bL \hookrightarrow \bL(\lambda_B(\underline{x})) \hookrightarrow \bL(\lambda_B(\underline{x}), \underline{x}^{h(B)}) $ are algebraic
and we can associate with $B$  two numbers:

\label{pagen}
\begin{itemize}
\item the degree of the field extension $\bL \hookrightarrow \bL(\lambda_B(\underline{x}))$ that we will denote $N(B)$,
\item the degree of the field extension $\bL(\lambda_B(\underline{x})) \hookrightarrow \bL(\lambda_B(\underline{x}), \underline{x}^{h(B)})$ that we will denote
$n(B)$.
\end{itemize}
 
\medskip
\noindent 
In this section we introduce the notion of \emph {compatibility} of a Weierstass polynomial $g$ 
with a pseudo-ball $B$. 
We define a polynomial $G_B(z)$ which will play an important role in the sequel. 

\begin{Definition} \label{compatible} 
Let $g(y)\in \bK[[\underline{x}]][y]$ be a Weierstrass polynomial and $B$ be a pseudo-ball
of finite height. If 
\begin{equation}\label{eq:comp}
g(\lambda_B(\underline{x}) + z\underline{x}^{h(B)})=G_B(z)\underline{x}^{q(g,B)}+\cdots
\end{equation}
for some $G_B(z)\in\bK[z]\setminus\{0\}$ and some exponent $q(g,B)\in (\mathbf Q_{\geq 0})^d$
then we will say that $g$~is compatible with $B$.  In \eqref{eq:comp} $\cdots$ means terms of higher order.
The polynomial $G_B(z)$ will be called the $B$-{characteristic polynomial} of $g$.
\end{Definition}

\begin{Example}

Return to Example \ref{ex:KL}.  Let $B=\alpha+x_{1}^{3/2}x_{2}\bK[[\underline{x}^{1/\bN}]]$ be a pseudo-ball of $T(f)$ of height $h(B)=\left(\frac{3}{2},1\right)$. 
Observe that 
\[
f(\lambda_{B}+z\underline{x}^{h(B)})=f(z\underline{x}^{\left(\frac{3}{2},1\right)})=z(z^{2}-1)x_{1}^{9/2}x_{2}^{3}+\cdots\]
Hence the polynomial $f$ is compatible with the pseudo-ball  $B$ and its $B$-characteristic polynomial is $F_B(z)=z(z^2-1)$, but for example   the polynomial
$g(y)=y-x_{1}-x_{2}$  is not compatible with $B$. 
\end{Example}

\noindent Our next goal  is to prove in Corollary \ref{CCKiel} that if a Weierstrass  polynomial is compatible with a pseudo-ball then any factor of it is  also  compatible.

\medskip

\begin{Lemma}\label{L:compatibility}
Let $g(y)\in \bK[[\underline{x}]][y]$ be a Weierstrass polynomial and let 
$B$ be a pseudo-ball of finite height.  
Consider $g(\lambda_B(\underline{x})+z\underline{x}^{h(B)})$ 
as a fractional power series $\tilde g(\underline{x})$  
with coefficients in $\bK[z]$. Then $g(y)$ is compatible with $B$ 
if and only if the Newton polytope of $\tilde g(\underline{x})$  
equals the Newton polytope of a monomial. \end{Lemma}

\noindent \begin{proof}
If $g$ is compatible with $B$ then by (\ref{eq:comp}) we get 
$\Delta\bigl(\tilde g(\underline{x})\bigr)=\Delta\bigl(\underline{x}^{q(g,B)}\bigr)$.
Conversely, suppose that  the Newton polytope of $\tilde g(\underline{x})$
equals the Newton polytope of the monomial $\underline{x}^{\bf q}$. 
Then $\tilde g(\underline{x})$ has a form $\underline{x}^{\bf q}\sum_{i=0}^n a_i(\underline{x})z^{n-i}$, 
where at least one of the values $a_i(0)$ is nonzero.
Hence the $B$-characteristic polynomial of $g$ is $G_B(z)= \sum_{i=0}^n a_i(0)z^{n-i}$.
\end{proof}

\begin{Remark}
\label{tt2}
From the proof of Lemma \ref{L:compatibility} we get that $\tilde g(\underline{x})$ has the form $G_B(z)\underline{x}^{q(g,B)}+ \sum_{h>q(g,B)}a_h(z)\underline{x}^h$, where  $a_h(z)\in\bK[z]$.
\end{Remark}

\begin{Corollary} 
\label{CCKiel}
Let $g\in \bK[[\underline{x}]][y]$ be a Weierstrass polynomial compatible with a pseudo-ball $B$. 
Then any factor of $g$ is compatible with $B$.
\end{Corollary}

\noindent \begin{proof}
The Newton polytope of the product is the Minkowki sum of Newton polytopes of the factors.
Hence, if $\Delta(\tilde g)=\Delta(\underline{x}^{\bf q})$ and $\tilde g=\tilde g_1\tilde g_2$
then $\Delta(\tilde g_i)$ have the form $\Delta(\underline{x}^{{\bf q}_i})$ for some ${\bf q}_1, {\bf q}_2$ such that ${\bf q}={\bf q}_1+{\bf q}_2.$
\end{proof}

\medskip
\noindent Next lemma generalizes to $d$ variables \cite[Lemma 3.1]{Forum}.

\begin{Lemma}\label{derivatives}
Let $f(y)\in \bK[[\underline{x}]][y]$ be a Weierstrass polynomial of degree $n$ compatible with the pseudo-ball $B$. Then for every $k\in\{1,\dots, \deg F_B(z)\}$ the Weierstrass polynomial 
$g(y)=\frac{(n-k)!}{n!}\frac{d^k}{dy^k}f(y)$ is also compatible with~$B$ 
and its $B$-characteristic polynomial is $G_B(z)=\frac{(n-k)!}{n!}\frac{d^k}{dz^k}F_B(z)$.
\end{Lemma}

\noindent \begin{proof} 
\noindent Differentiating  identity 
$f(\lambda_B(\underline{x}) + z\underline{x}^{h(B)})=F_B(z)\underline{x}^{q(f,B)}+\cdots$
with respect to $z$ we get 
$f'(\lambda_B(\underline{x})+z^{h(B)})\underline{x}^{h(B)}=F'_B(z)\underline{x}^{q(f,B)}+\cdots$.  
Hence $f'(\lambda_B(\underline{x})+z\underline{x}^{h(B)})=F'_B(z)\underline{x}^{q(f,B)-h(B)}+\cdots$,  which proves the lemma for $k=1$. 
The proof for higher derivatives runs by induction on $k$. 
\end{proof}

\medskip

\noindent  Let $f(y)\in \bK[[\underline{x}]][y]$ be a Weierstrass polynomial of degree $n$. The Weierstrass polynomial $\frac{(n-k)!}{n!}\frac{d^k}{dy^k}f(y)$ of Lemma \ref{derivatives} will be called the {\it normalized $k$th derivative} of the Weierstrass 
polynomial  $f(y)\in \bK[[\underline{x}]][y]$ and we will denote it by $f^{(k)}(y)$. The variety of equation $f^{(k)}=0$ is called the $k${\em th polar}  of $f=0$. Since  the normalized $n$th derivative of $f$ is constant, in the rest of the paper we consider normalized $k$th derivatives of $f$ for $1\leq k<\deg f$. 

\begin{Lemma}\label{subst}
Let $f(y)\in\bK[[\underline{x}]][y]$ be a  Weierstrass polynomial and 
let $B$ be a pseudo-ball of finite height.
\begin{enumerate}
\item If $f$ is compatible with $B$,  then for any 
$\gamma \in B$ we have
\begin{equation} 
\label{eq:F,q1xxx}
f(\gamma)=F_B(\lc_{B}\gamma)\underline{x}^{q(f,B)}+\cdots\,
\end{equation}
\item If $f(y)=\prod_{i=1}^n(y-\alpha_i)$ and 
we assume that one of the following holds: 
$\underline{x}$ is a single variable and $B$ is arbitrary 
or $f$ is quasi-ordinary and 
$B\in \tilde{T}(f)$
then $f$ is compatible with $B$ and we have

\begin{equation}
\label{eq:F,q2}
F_B(z)=\mbox{const}\prod_{i:\alpha_i\in B}(z-\lc_B\alpha_i)
\end{equation}
and 

\begin{equation}
\label{eq:F,q3}
q(f,B)=\sum_{i=1}^n \min(O(\lambda_B,\alpha_i),h(B)).
\end{equation}
\end{enumerate}
\end{Lemma} 

\noindent \begin{proof}
Since $\gamma \in B$ we can write $\gamma=\lambda_B(\underline{x})+z\underline{x}^{h(B)}$, where $z=\lc_B(\gamma)+\cdots$. By Remark \ref{tt2} we have
$f(\gamma)=f(\lambda_B(\underline{x})+z\underline{x}^{h(B)})=F_B(z)\underline{x}^{q(f,B)}+\cdots=F_B(\lc_{B}\gamma)\underline{x}^{q(f,B)}+\cdots$. This proves (\ref{eq:F,q1xxx}).

\medskip

\noindent Suppose $\gamma=\lambda_B(x)+zx^h,$ where $z$ is a constant. We have  $f(\gamma)=\prod_{i=1}^n(\gamma-\alpha_i)$. In order to prove (\ref{eq:F,q2}) and (\ref{eq:F,q3}),  
it is enough to compute the initial term of every factor $\gamma-\alpha_i$. If $\alpha_i\in B$ then the initial term of $\gamma-\alpha_i$ equals 
$(\lc_B\gamma-\lc_{B}\alpha_i)\underline{x}^{h(B)}$. Otherwise the initial terms of $\gamma-\alpha_i$ and $\lambda_B-\alpha_i$ are equal. We finish the proof multiplying the initial terms.
\end{proof}

\begin{Corollary}
Let $f(y)\in \bK[[\underline{x}]][y]$ be a quasi-ordinary Weierstrass polynomial. Then 
every factor of $f(y)$ is compatible with all 
bars $B\in T(f)$ of finite height. 
\end{Corollary}

\begin{Lemma}
\label{in-chain}
Let $f(y)\in \bK[[\underline{x}]][y]$ be a quasi-ordinary Weierstrass polynomial, $p(y)$ be a factor of $f(y)$ and $B,B' $ be bars of finite heights in $T(f)$ such that $B\perp B'$. Then 
\[
q(p,B')-q(p,B)=\sharp ( \Zer p \cap B')[h(B')-h(B)].
\]
\end{Lemma}
\noindent \begin{proof}
Put $p(y)=\prod_{\alpha \in \Zer p}(y-\alpha)$. Let $\gamma \in B, \gamma' \in B'$ be such that $O(\gamma, \alpha)=h(B)$  for all $\alpha \in B\cap \Zer p$ and
$\cont(\gamma', \alpha)=h(B')$ for all $\alpha \in B'\cap \Zer p$.   By the STI we get $O (\gamma',\alpha)=O(\gamma,\alpha)$ for any $\alpha \in \Zer p\backslash B'$. If $\alpha \in \Zer p \cap B'$  then $O(\gamma,\alpha)=h(B)$ and $O(\gamma',\alpha)=h(B')$. Hence

\begin{eqnarray*}
q(p,B')-q(p,B)&=&\sum_{\alpha \in \Zer p}O (\gamma',\alpha)-\sum_{\alpha \in \Zer p}O (\gamma,\alpha)\\
&=&\sharp (\Zer p \cap B')[h(B')-h(B)].
\end{eqnarray*}
\end{proof}

\noindent Lemma \ref{in-chain} is similar in spirit to \cite[Lemma 2.7]{Hungarica}.

\begin{Lemma}
\label{power}
Let $B$ be a quasi-ordinary pseudo-ball and let $g(y)\in \bK[[\underline{x}]][y]$ be a Weierstrass polynomial  
compatible with $B$. Then 
\begin{enumerate}
\item $G_{B}(z)=z^k\cdot H(z^{n(B)})$, for some $k\in \bN$ and $H(z)\in \bK[z]$. 
\item If $g$ is irreducible and quasi-ordinary then $G_{B}(z)=az^{k}$ or 
 $G_{B}(z)=a(z^{n(B)}-c)^{l}$, for some non-zero $a,c\in \bK$ and some $l\in \bN$.
 \end{enumerate}
\end{Lemma}

\noindent \begin{proof}
Let $\bL$ be the field of quotients of $\bK[[\underline{x}]]$.
By \cite[Lemma 5.7]{Lipman} and  \cite[Remark 2.7]{GP} the algebraic extension
$\bL(\lambda_B(\underline{x}))\hookrightarrow \bL(\lambda_B(\underline{x}),\underline{x}^{h(B)})$ is cyclic. Hence the generator  $\varphi$ of the group $Gal(\bL(\lambda_B(\underline{x})\hookrightarrow \bL(\lambda_B(\underline{x}),\underline{x}^{h(B)})$ acts as follows:
$\varphi(\lambda_B(\underline{x}))=\lambda_B(\underline{x})$ and 
$\varphi(\underline{x}^{h(B)})=\omega \underline{x}^{h(B)},$ where $\omega$ is a primitive $n(B)$th root of the unity.  Applying $\varphi$ to~(\ref{eq:comp}) we get
\begin{equation}
\label{KP1}
g(\lambda_B(\underline{x})+z \omega x^{h(B)})=G_B(z)\omega^kx^{q(g,B)}+\cdots
\end{equation}
for some $0\leq k < n(B)$. Substituting $\omega z$ for $z$ in (\ref{eq:comp}) 
and comparing with (\ref{KP1}) we get $G_B(z)\omega^k=G_B(\omega z)$. 
Multiplying this equality by $(\omega z)^{n(B)-k}$ and putting $W(z):=z^{n(B)-k}G_B(z)$ we obtain 
$W(z)= W(\omega z).$ This implies that $W(z)=\overline W(z^{n(B)})$, for some 
$\overline W(z)\in \bK[z]$. 
We finish the proof putting $H(z^{n(B)})=z^{-n(B)}\overline W(z^{n(B)})$. This proves the first part of the lemma.

\medskip
\noindent Suppose now that $g$ is irreducible and quasi-ordinary. Let $\gamma=\lambda_B(\underline{x})+ cx^{h(B)}+\cdots\in B\cap \Zer g$.
Since the extension $ \bL(\lambda_B(\underline{x}))\hookrightarrow \bL(\lambda_B(\underline{x}),\underline{x}^{h(B)})\hookrightarrow \bL(\gamma)$ is Galois,  any other root of $g$ belonging to $B$ has the form 
$\lambda_B(\underline{x})+ \omega^{i} cx^{h(B)}+\cdots$, for some $0\leq i < n(B)$. Using the first part of the lemma and the equality \eqref {eq:F,q2} we complete the proof.\end{proof}

\section{Conjugate pseudo-balls}
\label{section-Conjugate-pseudo-balls}

\noindent In this section we define an equivalence relation between pseudo-balls called {\em conjugacy relation}. This will allow us to introduce,  in Section \ref{section-Eggers-tree}, the notion of the Eggers tree of a quasi-ordinary Weierstrass polynomial. 

\medskip
\noindent Let $\bL$ be the field of fractions of $\bK[[\underline{x}]]$ and $\bM$ be the field of fractions of $\bK[[\underline{x}^{1/\bN}]]$.

\begin{Lemma} 
\label{propp:1}
Let $\varphi$ be an $\bL$-automorphism of $\bM$. Then
\begin{enumerate}
\item For any ${\bf q}\in \bQ^d$ there exists a root $\omega$ of the unity such that $\varphi(\underline{x}^{\bf q})=\omega \cdot \underline{x}^{\bf q}$,
\item $\varphi(\bK[[\underline{x}^{1/\bN}]]) = \bK[[\underline{x}^{1/\bN}]]$,
\item If $u$ is a unit  of the ring $\in \bK[[\underline{x}^{1/\bN}]]$ 
        and ${\bf q}\in (\bQ_{\geq 0})^d$ then
        $\varphi(u\cdot \underline{x}^{\bf q})=\tilde u \cdot \underline{x}^{\bf q}$ 
        for some unit $\tilde u\in \bK[[\underline{x}^{1/\bN}]]$. 
\end{enumerate}
\end{Lemma}

\noindent \begin{proof} 
\noindent Let  $k$ be a positive integer. Observe that
$x_i=\varphi(x_i)=\varphi\bigl((x_i^{1/k})^k\bigr)=\varphi(x_i^{1/k})^k$. 
Hence $\varphi(x_i^{1/k})=c\cdot x_i^{1/k}$ 
for some $c\in \bK\backslash\{0\}$ such that $c^k=1$. 
It follows that for any ${\bf q}\in\bQ^d$ there exists  $\omega\in \bK$  such that 
$\varphi(\underline{x}^{{\bf q}})=\omega\underline{x}^{\bf q}$ 
and $\omega^m=1$ for some positive integer $m$. 

\medskip 
\noindent
Every element of the ring $\bK[[\underline{x}^{1/\bN}]]$
can be represented as a finite sum 
$\sum_{\bf q} a_{\bf q} \underline {x}^{\bf q}$ 
where ${\bf q}=(q_1,\dots, q_d) \in(\bQ_{\geq 0})^d$ ($0\leq q_i<1$)
and $a_{\bf q }\in \bK[[\underline{x}]]$.
This together with~{\em 1}. proves items~{\em 2}.\ and~{\em 3}.\ of the lemma.
\end{proof}

\medskip

\noindent Let $B$, $B'$ be pseudo-balls. We say that $B$ and $B'$ are {\em conjugate} if 
there exists an $\bL$-automorphism $\varphi$ of   $\bM$ such that $B'=\varphi(B)$.  
The conjugacy of pseudo-balls is an equivalence relation. It follows from Lemma~\ref{propp:1} that conjugate pseudo-balls have the same height. 
Moreover two quasi-ordinary pseudo-balls $B$ and $B'$ of the same height are conjugate if any irreducible quasi-ordinary polynomial which has one of its roots in $B$ has another root in $B'$
(in this way conjugate bars were defined in \cite[Definition 6.1]{K-P}).
If $B'=\varphi(B)$ then $\lambda_{B'}=\varphi(\lambda_{B})$. The converse is also true;
if $h(B)=h(B')$ and there exists an $\bL$-automorphism $\varphi$ of $\bM$ such that $\lambda_{B'}=\varphi(\lambda_{B})$ then $B$ and $B'$ are conjugate.  \label{card[B]} It follows from the above 
that the number of pseudo-balls  conjugate with $B$ is equal to  the degree of the minimal polynomial of $\lambda_B$, which is the degree $N(B)$ of the field 
extension  $\bL \hookrightarrow \bL(\lambda_B(\underline{x}))$.

\begin{Lemma}
\label{LL:1}
Let   $B,B'$ be quasi-ordinary conjugate pseudo-balls.
If $p(y)\in \bK[[\underline{x}]][y]$ is a Weierstrass polynomial compatible with $B$  then 
\begin{enumerate}
\item $p(y)$ is compatible with $B'$. 
\item $q(p,B)=q(p,B')$.
\item The characteristic polynomials $P_{B'}(z)$ and $P_{B}(z)$ of $p(y)$ verify the equality
$P_{B'}(z)=\theta P_B(\omega z)$, for some roots of the unity $\theta$ and $\omega$.
\end{enumerate}
\end{Lemma}

\noindent \begin{proof} 
Let $\bL$ be the field of quotients of $\bK[[\underline{x}]]$ and let
$\varphi$ be a $\bL$-automorphism of $\bM$ such that $\varphi(B)=B'$. Then
$\varphi(\lambda_B)=\lambda_{B'}$. By Lemma \ref{propp:1} we have $\varphi(\underline{x}^{h(B)})=\omega^{-1} \underline{x}^{h(B)}$ and $\varphi(\underline{x}^{q(p,B)})=\theta \underline{x}^{q(p,B)}$ for some roots of the unity 
$\theta$ and $\omega$. 
Applying $\varphi$ to~(\ref{eq:comp}), with $g$ replaced by $p$, we get
\[
p(\lambda_{B'}+z \omega^{-1} \underline{x}^{h(B)})=P_B(z)\theta \underline{x}^{q(p,B)}+\cdots
\]

\medskip
\noindent This gives  $q(p,B)=q(p,B')$ and $P_{B'}(\omega^{-1} z)=\theta P_B(z)$. 
\end{proof}

\section{Kuo-Lu Lemma for higher derivatives}
\label{section-Kuo-Lu Lemma}
\noindent 
Let $f(y)\in \bK[[\underline{x}]][y]$ be a quasi-ordinary Weierstrass polynomial. 
We begin with combinatorial results concerning the Kuo-Lu tree $T(f)$. Remember that we identify any bar of $T(f)$ with the corresponding quasi-ordinary pseudo-ball.
At the end of the section we apply these results to Newton-Puiseux roots of higher derivatives of $f(y)$.

\medskip
\noindent Take an integer $k$ such that $1\leq k\leq \deg f$.  
With every bar $B$ of $T(f)$ we a\-sso\-ciate the numbers:
\begin{itemize}
\item $m(B)$ which is the number of roots of $f(y)$ which belong to $B$, 
\item $n_k(B)=\max\{m(B)-k,0\}$, and 
\item $t_k(B)=n_k(B)-\sum_{B\perp B'}n_k(B')$. 
\end{itemize} 

\begin{Remark}
\label{t=0}
\noindent For  $1\leq k < m(B)$ we have
$n_k(B)>0$, $t_k(B)>0$ and for $m(B)\leq k \leq \deg f$ we have
$n_k(B) =t_k(B)=0$.
\end{Remark}

\medskip
\noindent  We denote by $T_k(f)$ the  sub-tree of $T(f)$ consisting of the bars $B\in T(f)$ such that $m(B)\geq k$.

\medskip
\noindent Let $F\in\bK[z]$ be a non constant polynomial. 
Let $F^{(k)}$ denotes the $k$th derivative of $F$. 

\begin{Definition}
\noindent We will say that $F$ is $k$-{\em regular} 
if one of the following conditions holds:
\begin{enumerate}
\item $F^{(k)}$ is zero or
\item $F^{(k)}$ is nonzero and there is not a root of $F$ of multiplicity $\leq k$ which is a root of $F^{(k)}$. 
\end{enumerate}
\end{Definition}

\noindent 
Recall that common roots of a polynomial $F$ and its first derivative are multiple roots of $F$. 
Hence any polynomial is 1-regular. 

\medskip
\label{page:F irr}
\noindent In general it is not easy to verify the $k$-regular property. In this papers polynomials of the form
\begin{equation}
\label{eq:F irr}
F(z)=(z^{n}-c)^{l}\in K[z],
\end{equation}
play an important role. Their $k$-regularity, for any $k$, is a  consequence of Lemma \ref{AL}.

\begin{Remark}
\label{r: +-}
Let $F(z)=\mbox{const} \prod_{i=1}^r (z-z_i)^{m_i}$, where $z_i$ are pairwise different,  
$m_i\geq k$ for $1\leq i \leq s$ and $m_i < k$ for $s < i \leq r$.
After differentiating, the multiplicity of any root drops by one. Hence putting 
${\cal F}^{\oplus}(z)=\prod_{i=1}^s (z-z_i)^{m_i-k}$ we obtain the decomposition 
\begin{equation}\label{k-derivative}
F^{(k)}(z)={\cal F}^{\oplus}(z){\cal F}^{\ominus}(z)
\end{equation}
into two coprime polynomials. A polynomial  $F$ is $k$-regular if and only if $F$ and ${\cal F}^{\ominus}$ 
do not have common roots. 
\end{Remark}

\medskip

\begin{Definition}
Let $f\in \bK[[\underline x]][y]$ be a quasi-ordinary Weierstrass polynomial. We say that $f$ is  Kuo-Lu 
$k$-regular if for every $B\in T(f)$ of finite height the polynomial $F_B(z)$ is  $k$-regular. 
\end{Definition}

\medskip
\noindent We finish this subsection with some results for Weierstrass polynomials  with coefficients in the ring of the formal power series in one variable.

\medskip
\noindent Let $f(y)\in\bK[[x]][y]$ be a square-free Weierstrass polynomial. Fix $B\in T_k(f)$ and assume that $ \{B_1,\dots,B_s\}$ 
is the set of post-bars of $B$ in $T_k(f)$. Denote $B^{\circ}=B\setminus(B_1\cup\dots\cup B_s)$.

\medskip

\begin{Theorem}\label{higher-Kuo-Lu}
Let $f(y)\in\bK[[x]][y]$ be a square-free Weierstrass polynomial 
over the ring of formal power series in one variable. 
Let $f(y)= \prod_{i=1}^n(y-\alpha_i)$ and 
$f^{(k)}(y)=\prod_{j=1}^{n-k}(y-\beta_j)$ 
be the Newton-Puiseux factorizations of $f$ and $f^{(k)}$. Then

\begin{itemize}

\item[(i)]  for every $B\in T_k(f)$ the set $\{j: \beta_j\in B\}$ has $n_k(B)$ elements,

\item[(ii)] for every $B\in T_k(f)$ the set $\{j: \beta_j\in B^{\circ}\}$ has $t_k(B)$ elements,

\item[(iii)] for every $\beta_j$ there exists a unique $B\in T_k(f)$ such that $\beta_j\in B^{\circ}$.

\item[(iv)]  Let $B\in T_k(f)$. 
If the polynomial $F_B(z)$ is  $k$-regular 
then for every $\alpha_i\in B$, $\beta_j\in B^{\circ}$
one has $O(\alpha_i,\beta_j)=h(B)$. 
Otherwise there exist $\alpha_i\in B$, $\beta_j\in B^{\circ}$
such that $O(\alpha_i,\beta_j)>h(B)$.
\end{itemize}      
\end{Theorem}

\noindent \begin{proof}
\smallskip 
\noindent \textit{Proof of (i).}
Suppose first, that $B\in T_k(B)$ has finite height. 
Then by Lemma \ref{subst} $F_B(z) = \mbox{const} \prod_{i:\alpha_i\in B} (z-\lc_B(\alpha_i))$.
By equality $(4)$ of this lemma and Lemma \ref{derivatives} we get
$F_B^{(k)}(z) = \mbox{const} \prod_{j:\beta_j\in B} (z-\lc_B(\beta_j))$. 
Hence, the set $\{j:\beta_j\in B\}$ has $\deg F_B-k = n_k(B)$ elements. 

\medskip
\noindent If the height of $B$ is infinite then $B=\{\alpha_i \}$ 
for exactly one Newton-Puiseux root $\alpha_i$ of $f(y)$. 
Hence  for $k=1$ $n_1(B)=0$ and $f'(y)$ does not have roots in $B$ 
and for $k>1$ $B\notin T_k(f)$.

\smallskip 
\noindent \textit{Proof of (ii).} 
It is enough to count the elements of the set $\{j:\beta_j\in B^{\circ}\}$ using~\textit{(i)}.

\smallskip 
\noindent \textit{Proof of (iii).} 
Let $B_0$ be the root of the tree $T(f)$.  By~\textit{(i)},
$\{\beta_1,\dots, \beta_{n-k}\}$ is a subset of $B_0$.  
It is clear that the sets $B^{\circ}$ for $B\in T_k(f)$ are pairwise disjoint 
and their union is equal to $B_0$. This proves~\textit{(iii)}.

\smallskip 
\noindent \textit{Proof of (iv).} 
Assume that $B_1$, \dots, $B_r$ are the post-bars of $B$ 
supported at points $z_1$, \dots, $z_r$  respectively,
and that $m(B_i)\geq k$ for $i\in \{1,\dots s\}$, 
 $m(B_i)< k$ for $i\in \{s+1,\dots r\}$. Then by Lemma \ref{subst} $F_B(z)=\prod_{i=1}^r (z-z_i)^{m(B_i)}$. 

\noindent After Remark \ref{r: +-} the $k$th derivative of $F_B(z)$ is the product of two coprime polynomials 
$$F_B^{(k)}(z)={\cal F}^{\oplus}_B(z){\cal F}^{\ominus}_B(z),$$

\noindent where ${\cal F}^{\oplus}_B(z):=\prod_{i=1}^s (z-z_i)^{n_k(B_i)}$.

\medskip

\noindent We get $\deg {\cal F}^{\ominus}_B(z)=t_k(B)$. Hence it follows from~\emph{(ii)} 
and~\emph{(iii)} that all roots of ${\cal F}^{\ominus}_B(z)$ 
correspond to those Newton-Puiseux roots of $f^{(k)}(y)$ that belong to $B^{\circ}$.
For $\alpha_i\in B$, $\beta_j\in B^{\circ}$ one has $O(\alpha_i,\beta_j)>h(B)$ 
if and only if $\lc_B(\alpha_i)=\lc_B(\beta_j)$, which means that the polynomials $F_B(z)$ and 
${\cal F}^{\ominus}_B(z)$ 
have a common root. Since $F_B(z)$ is $k$-regular if and only if ${\cal F}^{\ominus}_B(z)$ and $F_B(z)$ do not have common roots we get~\textit{(iv)}.
\end{proof}

\begin{Remark}\label{R:1}
Let $f(y)\in\bK[[x]][y]$ be a square-free Weierstrass polynomial 
over the ring of formal power series in one variable. 
Let $B\in T_k(f)$, $\beta_i \in B\cap \Zer f^{(k)}$ and put $c=\lc_{B}\beta_i$. 
Then ${\cal F}^{\oplus}_B(c)\neq 0$ if and only if $\beta_i \in B^{\circ}$. 
If ${\cal F}^{\oplus}_B(c)=0$ then there exists a sequence of postbars $B\perp_c B_1\perp \cdots \perp B_l$  such that 
$\beta_i \in B_l^{\circ}$ and $ B_l\in T_k(f)$.
\end{Remark}

\medskip

\noindent For quasi-ordinary Weierstrass polynomials which are Kuo-Lu $k$-regular, the counterpart of \cite[Lemma 3.3]{K-L} is true:

\begin{Corollary}\label{Kuo-Lu}
Let $f(y)\in\bK[[x]][y]$ be a square-free Weierstrass polynomial 
over the ring of formal power series in one variable. 
Assume that $f$ is Kuo-Lu $k$-regular. Then under assumptions and notations of Theorem~\ref{higher-Kuo-Lu}, for  every $\alpha_i\in \Zer f$, $\beta_s\in \Zer f^{(k)}$ there exists $\alpha_j\in \Zer f$ such that 
$O(\alpha_i,\beta_s)=O(\alpha_i,\alpha_j)$.
\end{Corollary}

\section{Newton polytopes of resultants}
\label{section-Newton-polytopes-resultants}

\noindent In this section we give a formula for the Newton polytope of the resultant 
$\Res_y(f^{(k)}(y),p(y)-T)$, where $f(y)$ is a Kuo-Lu $k$-regular quasi-ordinary 
Weierstrass polynomial, $p(y)$ is a factor of $f(y)$ and $T$ is a new variable.  We prove that for irreducible $p(y)$,  the Newton polytope of the resultant is polygonal.

\subsection{Monomial substitutions}

\noindent Let $g(\underline x,y)\in \bK[[\underline x,y]]$. For any monomial substitution 
$x_1=u^{r_1}$,\dots, $x_d=u^{r_d}$, where $r_i$ are positive integers, we put
\begin{equation}
\label{bar}
\bar g^{[{\bf r}]}(u,y):=g(u^{r_1},\ldots,u^{r_d},y).
\end{equation}

\noindent We will write simply $\bar{g}(u,y)$  when no confusion can arise.

\medskip

\noindent Observe that for $g=\underline{x}^{\bf s}$  we get $\bar g^{[{\bf r}]}=u^{\langle \bf{r},\bf{s}\rangle}$, where $\langle \cdot,\cdot \rangle$ denotes the scalar product.

\begin{Lemma}\label{L:2}
Let $f(y)\in \bK[[\underline x]][y]$ be a quasi-ordinary Weierstrass polynomial. There is  a one-to-one correspondence between the bars of $T(f)$ and the bars of $T(\bar f^{[{\bf r}]})$. 
If $B$ and $\bar B$ are the corresponding bars of $T(f)$ and $T(\bar f^{[{\bf r}]})$ respectively then 
\begin{enumerate}
\item $h(\bar B) = \scalar{{\bf r}}{h(B)}$ and  $t_k(\bar B)=t_k(B)$.
\item For any factor $g$ of $f$, the $B$-characteristic polynomial of $g$ and the $\bar B$-characteristic polynomial of $\bar g^{[{\bf r}]}$ are equal and $q(\bar g^{[{\bf r}]},\bar B) = \scalar{{\bf r}}{q(g,B)}$.

\end{enumerate}
\end{Lemma}

\noindent \begin{proof}
Set $u^{\bf r}=(u^{r_1},\dots,u^{r_d})$. 
If  $\Zer f=\{\alpha_i(\underline{x})\}_{i=1}^n$ then $\Zer \bar f^{[{\bf r}]}=\{\alpha_i(u^{\bf r})\}_{i=1}^n$
and $O(\alpha_i(u^{\bf r}),\alpha_j(u^{\bf r}))=\scalar{{\bf r}}
{O(\alpha_i(\underline{x}),\alpha_j(\underline{x}))}$ for $i\neq j$.

\noindent Hence every bar $B=\{\alpha_{i_j}(\underline{x})\}_{j=1}^k$ of $T(f)$ yields the bar 
$\bar B=\{\alpha_{i_j}(u^{{\bf r}})\}_{j=1}^k$ of $T(\bar f^{[{\bf r}]})$ of height 
$\scalar{{\bf r}}{h(B)}$. 

\medskip

\noindent Substituting $u^{r_i}$ for $x_i$ in the equation (\ref{eq:comp})  appearing in  Definition \ref{compatible},  we get 

\[
\bar g^{[{\bf r}]}(\lambda_{\bar B}(u) + zu^{h(\bar B)})=G_{B}(z)u^{\scalar{{\bf r}}{q(g,B)}}+\cdots,
\]
\noindent hence the second part of the lemma follows.
\end{proof}

\noindent The proof of the next lemma is similar in spirit to the proof of \cite[Theorem 4.1]{IMRN} 
and the proof of \cite[Theorem 9.2]{IMRN}.  The same arguments were used there in special situation. Here we repeat the proof for the convenience of the reader.

\begin{Lemma}
\label{R2}
Let $g(\underline x,y)\in \bK[[\underline x,y]]$ and $\Delta\subseteq \bR^{d+1}$ be a Newton polytope.  For any ${\bf r}\in (\bR_{>0})^d$ let $\bar\Delta^{[{\bf r}]}$ be the image of $\Delta$ by the linear mapping 
$\pi_{\bf r}:\bR^{d}\times \bR \longrightarrow \bR^2$ given by  
$({\bf a}, b)\mapsto (\langle {\bf r}, {\bf a} \rangle, b)$. 
If $\Delta(\bar g^{[{\bf r}]})=\bar{\Delta}^{[{\bf r}]}$  for every ${\bf r}\in (\bN\backslash\{0\})^d$, then $\Delta(g)=\Delta$. 
\end{Lemma}

\noindent \begin{proof} For every Newton polytope $\Delta \subseteq (\bR_{\geq 0})^{d+1}$ and every $v \in (\bR_{\geq 0})^{d+1}$ we define the {\em support function} $l(v,\Delta)=\min\{\langle v,\alpha\rangle \;:\;\alpha \in \Delta\}$. To prove the lemma it is enough to show that the support functions $l(\cdot,\Delta(g))$ and  $l(\cdot,\Delta)$ are equal. As these functions are continuous 
it suffices to show the equality on a dense subset of $\bR_{\geq0}^{d+1}$.

\noindent Let $\vec r=(r_1,\dots,r_{d+1})=({\bf r},r_{d+1})\in\bR_{\geq0}^{d+1}$, where
${\bf r}=(r_1,\dots,r_d)$.\\

\noindent Perturbing $\vec r$ a little we may assume that the hyperplane 
$\{\,\alpha\in\bR^{d+1}:\scalar{\vec r}{\alpha}=l(\vec r,\Delta(g)\,\}$ 
supports $\Delta(g)$ at exactly one point $\check{\alpha}=(\underline{\check{\alpha}},\check{\alpha}_{d+1})$. 
Since after a small change of $\vec r$ the support point remains the same, we can assume,
perturbing  $\vec r$ again if necessary, that all $r_i$ are positive rational numbers.

\noindent We will show that 
\begin{equation}\label{Eq:3}
l(\vec r,\Delta)=l(\vec r,\Delta(g)).
\end{equation}

\noindent Multiplying $\vec r$ by the common denominator of $r_1$, \dots, $r_{d+1}$ we may 
assume that all $r_i$ are positive integers. At this point of the proof 
we fix $\vec r$.  
We claim that
$l(\vec r,\Delta)=l\bigl((1,r_{d+1}),\bar{\Delta}^{[{\bf r}]}\bigr)$ and 
$l(\vec r,\Delta(g))=l\bigl((1,r_{d+1}),\Delta\left (\bar g^{[{\bf r}]}\right )\bigr)$.

\medskip

\noindent First equality follows from the definition of $\pi_{\bf r}$ and the
identity 
\[
\scalar{\vec r}{\alpha}=\scalar{(1,r_{d+1})}{\pi_{\bf r}(\alpha)}
\] 

\noindent for $\alpha\in\bR^{d+1}.$

\medskip
\noindent  Write $\alpha=(\underline \alpha, \alpha_{d+1})\in \bR^{d+1}$ and $g(\underline x,y)=\sum_{\alpha}d_{\alpha}\underline {x}^{\underline {\alpha}}y^{\alpha_{d+1}}\in \bK[[\underline x,y]]$. Since the hyperplane 
$\{\,\alpha\in\bR^{d+1}:\scalar{\vec r}{\alpha}=l(\vec r,\Delta(g)\,\}$ 
supports $\Delta(g)$ at $\check{\alpha}$, the term 
$d_{\check{\alpha}}u^{\scalar{{\bf r}}{\underline{\check{\alpha}}}}y^{\check{\alpha}_{d+1}}$ of $\bar g^{[{\bf r}]}$, satisfies the equality 
$\scalar{{\bf r}}{\underline{\check{\alpha}}}+r_{d+1}\check{\alpha}_{d+1}=l(\vec r,\Delta(g))$, 
while for all other terms  $d_{\alpha}u^{\scalar{{\bf r}}
{\underline{\alpha}}}y^{\alpha_{d+1}}$ with $d_{\alpha}\neq0$ appearing in  $\bar g^{[{\bf r}]}$,
we have  $\scalar{{\bf r}}{\underline{\alpha}}+r_{d+1}\alpha_{d+1}>l(\vec r,\Delta(g))$.

\medskip

\noindent 
Hence 
$l\bigl((1,r_{d+1}),\Delta(g)\bigr)=
\scalar{{\bf r}}{\underline{\check{\alpha}}}+r_{d+1}\check{\alpha}_{d+1}=
l(\vec r,\Delta(g))$, so we get (\ref{Eq:3}). 
\end{proof}

\begin{Corollary}
\label{R1}
Let $g_1(\underline x,y)$,  $g_2(\underline x, y)\in \bK[[\underline x,y]]$. Suppose that 
$\Delta(\bar g_1^{[{\bf r}]})=\Delta(\bar g_2^{[{\bf r}]})$ for every ${\bf r}\in (\bN\backslash\{0\})^d$. Then  $\Delta(g_1)=\Delta(g_2)$.
\end{Corollary}

\begin{Theorem}\label{higher-discriminants}
Assume that $f\in\bK[[\underline{x}]][y]$ is a Kuo-Lu $k$-regular quasi-ordinary Weierstrass polynomial
and $p$ is a Weierstrass polynomial which is a factor of $f$ in $\bK[[\underline{x}]][y]$. 
Then the Newton polytope of $R(T):=\Res_y(f^{(k)}(y),p(y)-T)\in \bK[[\underline{x}]][T]$ is  equal to 
\begin{equation}\label{jac_Nd}
\sum_{\genfrac{}{}{0pt}{2}{B\in T(f)}{t_k(B)\neq 0}} 
\left\{\Teissr{t_k(B)q(p,B)}{\rule{0pt}{2.5ex} t_k(B)}{13}{6.5}\right\}.
\end{equation}
\end{Theorem}

\noindent \begin{proof}
\noindent First we will prove the theorem for $d=1$. We use the notation of Theorem \ref{higher-Kuo-Lu}. Let $\prod_{j=1}^{n-k}(y-\beta_j)$  be the Newton-Puiseux factorization of $f^{(k)}(y)$. By the well-known properties of the resultants we have
\begin{equation}
\label{Ress}
\Res_y(f^{(k)}(y),p(y)-T)=\pm \prod_{j=1}^{n-k}(p(\beta_j)-T).
\end{equation}

\noindent 
By Theorem~\ref{higher-Kuo-Lu}, for every $\beta_j$ there exists a unique bar $B\in T(f)$ such that  $\beta_j\in B^{\circ}.$ For such a bar,  $h(B)$ is finite and $t_k(B)\neq 0$.
By Corollary~\ref{CCKiel} the polynomial $p$ is compatible with $B$ and by (\ref{eq:F,q2}) of Lemma \ref{subst} ${P}_{B}(z)$ is a factor of ${F}_{B}(z)$. 
By Theorem  \ref{higher-Kuo-Lu} $(iv)$ we get that $O(\alpha_i,\beta_j)=h(B)$ for any $\alpha_i\in B$. Hence $\lc_B\beta_j$ does not belong to the set $\{\lc_B\alpha_i\;:\;\alpha_i\in B\}$. So by the equality (\ref{eq:F,q2}) in Lemma \ref{subst} we have ${ F}_B(\lc_B\beta_j)\neq 0$ and consequently ${P}_{B}(\lc_B\beta_j)\neq 0$.
Now, using  equality (\ref{eq:F,q1xxx}) of Lemma \ref{subst} we conclude that the Newton polytope of $p(\beta_j)-T$ is equal to 
 $\left\{\Teissr{q(p,B)}{1}{7}{3.5}\right\}$.
 \noindent Using the property that  the Newton polytope of a product is the Minkowski sum of the Newton polytopes of its factors, and $(ii)$ of Theorem  \ref{higher-Kuo-Lu} we finish the proof for $d=1$.
 
\medskip
\noindent Assume now that $d>1$.  

\medskip 
\noindent Let  $x_1=u^{r_1}$,\dots, $x_d=u^{r_d}$ be a monomial substitution, where $r_i$ are positive integers. By Lemma ~\ref{L:2} $ f^{[{\bf r}]}$ is Kuo-Lu $k$-regular, hence by the first part of the proof ($d=1$)
\[
\Delta(\bar R^{[{\bf r}]})=
\sum_{\genfrac{}{}{0pt}{2}{B\in T(f)}{t_k(B)\neq 0}} 
\left\{\Teissr{t_k(\bar B)q(\bar p^{[{\bf r}]},\bar B)}{\rule{0pt}{2.5ex} t_k(\bar B)}{14}{7}\right\}.
\]

\noindent For any elementary polytope of the above sum, Lemma \ref{L:2} gives 

\[
\left\{\Teissr{t_k(\bar B)q(\bar p^{[{\bf r}]},\bar B)}{\rule{0pt}{2.5ex} t_k(\bar B)}{14}{7}\right\} = 
\left\{\Teissr{t_k(B)\langle {\bf r},q(p,B)\rangle}{\rule{0pt}{2.5ex} t_k(B)}{16}{8}\right\} = 
\pi_{\bf r}\left(\left\{\Teissr{t_k(B)q(p,B)}{\rule{0pt}{2.5ex} t_k(B)}{13}{6.5}\right\}\right).
\]
Since the image of the Minkowski sum of Newton polytopes is the Minkowski sum of the images, we get 
$\Delta(\bar R^{[{\bf r}]})=\pi_{\bf r}(\Delta)$, where $\Delta$ denotes the Newton polytope given in (\ref{jac_Nd}). By Lemma~\ref{R2} we get
$\Delta(R)=\Delta$.
\end{proof}

\section{ Eggers tree of a quasi-ordinary Weierstrass polynomial}
\label{section-Eggers-tree}

\noindent In this section we introduce the {\em Eggers tree} of a quasi-ordinary Weierstrass polynomial $f$, after the conjugacy relation defined in Section \ref{section-Conjugate-pseudo-balls}. Denote by  $[B]$ the conjugacy class of the pseudo-ball $B$  of the Kuo-Lu $T(f)$. By definition, the {\em Eggers tree} of $f$, denoted by $E(f)$, is the set of conjugacy classes with the natural order induced by the Kuo-Lu tree. This is the natural generalization of the Eggers tree associated with plane curves in \cite{Eggers}. The notion of Eggers tree, for quasi-ordinary singularities, was introduced by Popescu-Pampu in \cite{tesis}. He defined a slightly different notion of the Eggers tree, since he generalized  to quasi-ordinary singularities the version of Eggers tree defined for curves in \cite{Wall}.

\medskip
\noindent The leaves of $E(f)$ correspond with irreducible factors of $f$. Following Eggers we draw them in white color. By definition, the {\em root} of $E(f)$ is its vertex of minimum height. The {\em branches} of $E(f)$ are the smallest sub-trees of $E(f)$ containing the root and one of its leaves. Let $[B]$ be a vertex in the branch of $E(f)$ corresponding with the irreducible componente $f_i$ of $f$. Eggers draws in a dashed way the edge leaving from  the vertex $[B]$ in this branch if there are not two roots of $f_{i}$ with contact  $h(B)$.

\medskip
\noindent Recall  that the number of pseudo-balls conjugate with  a quasi-ordinary pseudo-ball $B$ is $N(B)$ (see page \pageref{card[B]}).

\medskip
\noindent Let $[B]$ be a vertex of the Eggers tree of a quasi-ordinary polynomial $f$. By Lemma \ref{LL:1}, for any $k\in \{1,\ldots, \deg f\}$,  the numbers $n_{k}(B)$ and $t_{k}(B)$ do not depend on the representative of $[B]$. Moreover, if $p(y)\in \bK[[\underline{x}]][y]$ is a Weierstrass polynomial  compatible with $B$ then the number $q(p,B)$  and the degree of its $B$-characteristic polynomial are also independent of the representative of $[B]$. 

\medskip

\label{Egg ex}
\noindent 
The Eggers tree of the quasi-ordinary polynomial $f=f_{1}f_{2}$ from Example \ref{ex:KL} is 
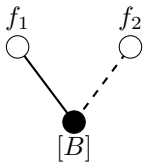
\begin{center}
\begin{tikzpicture}[scale=0.5]
\draw [thick](0,0) -- (-1.5,2); 
\draw[thick, dashed](0,0) -- (1.5,2);
\draw (0,-0.1) node[below]{$[B]$};
\draw (-1.5,2.2) node[above]{$f_{1}$};
\draw (1.5,2.2) node[above]{$f_{2}$};
\node[draw,circle,inner sep=3pt,fill=black] at (0,0) {};
\node[draw,circle,inner sep=3pt,fill=white] at (1.5,2) {};
\node[draw,circle,inner sep=3pt,fill=white] at (-1.5,2) {};
\end{tikzpicture}
\end{center}

\medskip

\begin{Remark}
\label{incr}
If $p$ is an irreducible factor of $f$ then, following Lemma \ref{in-chain}, the sequence $\{q(p,B)\}_{[B]}$ is increasing along the branch ${P}$ of the Eggers tree of $f$ containing the leave representing $p$. Moreover, if $[B]$ does not belong to $P$ then $q(p,B)=q(p,B_{0})$, where $[B_{0}]$ is the last common vertex of $P$ and the branches of the Eggers tree containing $[B]$. Hence, the set $\{q(p,B)\}_{[B]}$ is well-ordered.
\end{Remark}

\noindent After Remark \ref{incr} we get

\begin{Corollary}
\label{coro:polygonal}
Let $f\in\bK[[\underline{x}]][y]$ be a Kuo-Lu $k$-regular quasi-ordinary Weierstrass polynomial
and $p$  a Weierstrass polynomial which is an irreducible factor of $f$ in $\bK[[\underline{x}]][y]$. Then the Newton polytope in (\ref{jac_Nd}) is polygonal.
\end{Corollary}

\section{Irreducible factors of higher derivatives}
\label{section-Irreducible factors}

Let $f$ be a quasi-ordinary Weierstrass polynomial. In this section we 
study irreducible factors of normalized higher derivatives $f^{(k)}$. 
We show that every such an irreducible factor can be associated with  a certain vertex $[B]$ of the Eggers tree of $f$.  By definition an {\em Eggers factor} will be the product of all irreducible factors associated with the same vertex of $E(f)$. The {\em Eggers factorization} of a higher derivative is the product of all its Eggers factors. It generalizes to higher derivatives the factorization of the first polar given in \cite{Eggers} and \cite{GB} for  plane curves and in \cite{GB-GP} for quasi-ordinary polynomials.\\

\noindent Let $F_B(z)$ be the $B$-characteristic polynomial of $f$. After Remark \ref{r: +-}, the polynomial $F_B^{(k)}(z)$ is the product of two coprime polynomials
${\cal F}^{\oplus}_B(z)$ and ${\cal F}^{\ominus}_B(z),$
where 
\[
{\cal F}^{\oplus}_B(z)=\prod_{B\perp_{z_i} B_i} (z-z_i)^{n_k(B_i)}.
\]

\begin{Theorem}
\label{pack}
Let $f(y)$ be a quasi-ordinary Weierstrass polynomial
and let $g(y) \in \bK[[\underline{x}]][y]$ be a Weierstrass polynomial which 
is an irreducible factor of $f^{(k)}(y)$. 
Then there exists $[B]\in E(f)$, with $B\in T_k(f)$,  such that:
\begin{enumerate}
\item If $B'\in T_k(f)\backslash [B]$ then every root of $G_{B'}(z)$ 
        is a root of ${\cal F}^{\oplus}_{B'}(z)$.
\item If $B'\in T_k(f)\cap [B]$ then $G_{B'}(z)$ 
         and  ${\cal F}^{\oplus}_{B'}(z)$ do not have common roots.
         Moreover 
\begin{equation}
\label{fact}
G_B(z)=az^l\;\; \hbox{\rm or }\;\; G_B(z)=a(z^{n(B)}-c)^l
\end{equation}
\noindent for some $l\geq 1$ and $a,c\in \bK\backslash \{0\}$. 
If $l=1$ then $g(y)$ is quasi-ordinary.
\end{enumerate}
\end{Theorem}

\noindent \begin{proof} 
Let 
${\cal T}=\{\,B\in T_k(f): G_B(z) \mbox{ has a root which is not a root of } {\cal F}^{\oplus}_B(z)\,\}$.  
By Remark \ref{R:1}, $B\in {\cal T}$ if and only if for any monomial substitution $\bar g$ 
has a Newton-Puiseux root that belongs to $\bar B^{\circ}$.

\medskip
\noindent Let ${\cal E}=\{\,[B]\in E(f): B\in {\cal T}\,\}$. We will show that ${\cal E}$ has only one element. Suppose that this is not the case,
and let $[B_0]$ be the infimum of ${\cal E}$ in the ordered set $E(f)$ 
(the infimum exists because $E(f)$ has the structure of a tree).  

\medskip
\noindent 
Let $[B']$ be an element of ${\cal E}$ different from $[B_0]$ and let $p$ 
be any irreducible factor of $f$ such that one of its roots belongs to $B'$. 
By definition of the Eggers tree there exists $B_1\in [B_0]$ such that $B'\subsetneq B_1$. Since $B'\in T_k(f)$, 
one of the roots of $P_{B_1}(z)$ has multiplicity bigger or equal than $k$. Hence,  
by  the second statement of Lemma \ref{power} all the roots of $P_{B_1}(z)$ have this property. By  \eqref{k-derivative}, the polynomial $P_{B_1}(z)$ could only share roots with ${\cal F}^{\oplus}_{B_1}(z)$. Hence, by Remark \ref{r: +-} the polynomials $P_{B_1}(z)$ and ${\cal F}^{\ominus}_{B_1}(z)$ are coprime.

\medskip
\noindent Let $\bar g=\prod_{i=1}^m(y-\bar\beta_i)$ be the Newton-Puiseux factorization of $g$ after some 
monomial substitution.  Fix $B\in [B_0]$. By  Lemmas \ref{LL:1} and \ref{L:2} we get $q(\bar p,\bar B)=q(\bar p,\bar B_0)$. 
Let us define two sets of indexes associated with $\bar B$:
$$ I_{\bar B}=\{i: \bar\beta_i\in \bar B, P_B(\lc_{\bar B}\bar\beta_i)\neq 0\,\}, $$
$$ J_{\bar B}=\{i: \bar\beta_i\in \bar B, P_B(\lc_{\bar B}\bar\beta_i)= 0\,\}. $$

\medskip
\noindent 
Directly from the definition of  $P_B$ we have: 
if $i\in I_{\bar B}$ then $\ord \bar{p}(\bar\beta_i)=q(\bar p,\bar B_0)$, 
and if $i\in J_{\bar B}$ then  $\ord \bar{p}(\bar\beta_i)>q(\bar p,\bar B_0)$. 

\medskip
\noindent
The cardinality of $I_{\bar B}$ is equal to the number of roots of $G_B(z)$ counted with multiplicities
which are not the roots of $P_B(z)$. Similarly
the cardinality of $J_{\bar B}$ is equal to the number of roots of $G_B(z)$ counted with multiplicities
which are the roots of $P_B(z)$. Hence the cardinality of these sets does not depend 
on the choice of the monomial substitution. 
Let $I:=\bigcup_{B\in [B_0]} I_{\bar B}$ and $J:=\bigcup_{B\in [B_0]} J_{\bar B}$. Observe that 
 $\ord \bar{p}(\bar\beta_i)=q(\bar p,\bar B_0)$ for $i\in I$ and 
 $\ord \bar{p}(\bar\beta_i)>q(\bar p,\bar B_0)$ for $i\in J$.
 
\medskip
\noindent
The sets $I$ and $J$ depend on the choice of the monomial substitution but their cardinality does not. We will show that the set $J$ is nonempty. 
Since $B'\in {\cal T}$, there exists $\bar\beta_i \in \bar B'$.
Any root of $\bar p$ that belongs to $\bar B'$ has the same leading 
coefficient 
with respect to  $\bar B_1$ as $\bar\beta_i$. 
Hence  $P_{B_1}(\lc_{\bar B_1}\bar\beta_i)=0$, which gives  $i\in J_{\bar{B}_1}\subset J$.


\medskip
\noindent
Now we will prove that the set $I$ is empty. Suppose that it is not the case. Put $R(T):= \Res_y(g, p-T)$ and $\bar R(T):= \Res_y(\bar g, \bar p-T)$. We can write
\begin{eqnarray*}
R(T)        &=& \pm T^m+c_1T^{m-1}+\cdots +c_m, \\
\bar R(T) &=& \pm T^m+\bar{c}_1T^{m-1}+\cdots +\bar{c}_m,
\end{eqnarray*} 

\noindent for some $c_i\in \bK[[\underline{x}]]$. 
By a well-known formula for the resultant we have $\bar R(T)=\pm\prod_{i=1}^m(\bar p(\bar\beta_i)-T)$. Since the Newton polygon of a product is the Minkowski sum of the Newton polygons of its factors, $\Delta(\bar R(T))$ has an edge of inclination $q(\bar p,\bar B_{0})$
starting in the point $(0,m)$. The projection of this edge to the vertical axis has length $\sharp I$. This gives
\begin{eqnarray*}
\ord \bar{c}_i&\geq& iq(\bar p,\bar B_{0}) \;\;\hbox{\rm for $1\leq i < \sharp I$}, \nonumber \\ 
\ord \bar{c}_i&=& iq(\bar p,\bar B_{0}) \;\;\hbox{\rm  for $i=\sharp I$},\\ 
\ord \bar{c}_i&>& iq(\bar p,\bar B_{0})  \;\;\hbox{\rm for $  \sharp I<i\leq m$}. \nonumber
\end{eqnarray*}

\noindent Since the monomial substitution was arbitrary, 
we have
\begin{eqnarray*}
{c}_i\bK[[\underline{x}^{1/\bN}]]&\subseteq& \underline{x}^{iq (p,B_{0})}\bK[[\underline{x}^{1/\bN}]] \;\;\;\;\;\;\hbox{\rm for $1\leq i < \sharp I$},  \\
{c}_i\bK[[\underline{x}^{1/\bN}]]&=& \underline{x}^{iq (p,B_{0})}\bK[[\underline{x}^{1/\bN}]]
 \;\;\;\;\;\;\hbox{\rm  for $i=\sharp I$}, \\ 
{c}_i\bK[[\underline{x}^{1/\bN}]]&\subsetneq& \underline{x}^{iq (p,B_{0})}\bK[[\underline{x}^{1/\bN}]]   \;\;\;\;\;\;\hbox{\rm for $  \sharp I<i\leq m$}.
\end{eqnarray*}

\medskip
\noindent
By Corollary~\ref{irred} 
$g$ is not irreducible and we get a contradiction. 

\medskip
\noindent
We conclude that $I=\emptyset$. This means that for every $\bar\beta_i$
there exists $B\in[B_0]$ such that $\bar\beta_i\in\bar B$ and 
$\ord \bar{p}(\bar\beta_i)>q(\bar p,\bar B)$. By Remark \ref{R:1}, $\bar \beta_i$ belongs to a post-bar of $\bar B$, 
which has a nonempty intersection with $\Zer \bar p$.
All post-bars of $B\in[B_0]$ that have nonempty intersection with $\Zer p$ conjugate.  
They form the vertex of $E(f)$, bigger than $[B_0]$,  which is smaller or equal (with the natural order in $E(f)$) than any element of $\cal E$.
Hence $[B_0]$ cannot be the infimum of ${\cal E}$ and we arrive again at a contradiction. 

\medskip
\noindent
We have shown that ${\cal E}$ has only one element. Denote it  by  $[B_0]$.  Hence  for any monomial substitution we have 
$\Zer \bar g \subset \bigcup_{B\in [B_0]}\bar B^{\circ}$.
By Remark \ref{R:1} we get~\textit{1.}\  and the first part of~\textit{2.}
 
\medskip
\noindent
Now we will find the form of $G_B(z)$, for any $B\in [B_0]$. 
If for every $\bar\beta_i\in B$ the leading coefficient $\lc_{\bar B} \bar\beta_i$ is $0$,  then obviously $G_B(z)=az^l$. 
Otherwise by Lemma \ref{power} there exist $c\neq 0$ and a polynomial $G_1(z)$ 
coprime with $z^{n(B)}-c^{n(B)}$ such that $G_B(z)=G_1(z)(z^{n(B)}-c^{n(B)})^l$.  
Let $p(y)$ be the minimal Weierstrass polynomial of  $\lambda_B(\underline{x})+c\underline{x}^{h(B)}$. 
Then $P_B(z)=\mbox{const}\cdot(z^{n(B)}-c^{n(B)})$.

\medskip
\noindent Proceeding as in the first part of the proof we define again the sets $I$, $J$ of indexes. 
By the choice of $p(y)$ the set $J$ is nonempty. If the polynomial $G_1(z)$ has positive degree 
then the set $I$ is nonempty and we arrive at a contradiction. Hence $G_1(z)$ is a constant 
which proves the second part of the theorem.

\medskip
\noindent
Now we prove that if $l=1$ in (\ref{fact}) then $g(y)$ is quasi-ordinary. 
Let $p(y)\in  \bK[[\underline{x}]][y]$ be  the minimal polynomial of $\lambda_B(\underline{x})$ if $G_B(z)=az$ or the minimal polynomial of $\lambda_B(\underline{x})+c{\underline{x}}^{h(B)}$ 
if $G_B(z)=a(z^{n(B)}-c^{n(B)})$.
Then  $G_B(z)$ is equal to $P_B(z)$ up to multiplication by a constant. 
By Lemma~\ref{LL:1} for any $B'\in[B_{0}]=[B]$
the characteristic polynomials $G_{B'}(z)$ and $P_{B'}(z)$ have the same form, 
in particular have the same number of roots and all their roots are simple. 
Take any monomial substitution and let $\bar\beta' $, $\bar\beta''$ be different roots of $\bar g(y)$. 
Since $\Zer \bar g \subset \bigcup_{B\in [B_0]}\bar B^{\circ}$ there exist $B', B''\in [B_0]$ 
such that $\bar\beta'\in\bar B'$ and $\bar\beta''\in\bar B''$.
If $B'=B''$ then $O(\bar\beta',\bar\beta'')=h(\bar B')$ because $\bar\beta'$ and $\bar\beta''$
have different leading coefficients with respect to $\bar B'$. 
If $B'\neq B''$ then $O(\bar\beta',\bar\beta'')=O(\lambda_{\bar B'},\lambda_{\bar B''})$.
In both cases the contact $O(\bar\beta',\bar\beta'')$ depends only on $B'$ and $B''$. 
The same argument applies to the roots of $\bar p(y)$. 
As a consequence any bijection $\Phi:\Zer \bar g\to \Zer \bar p$ such that 
$\Phi(\bar B'\cap\Zer \bar g)=\bar B'\cap\Zer \bar p$ for $B'\in [B_0]$ preserves contacts. 

\medskip

\noindent Since the discriminant of a monic polynomial is the product of differences of its roots, 
the discriminant of $\bar{g}$ and the discriminant of $\bar{p}$ have the same order.
Then by Corollary~\ref{R1}
  the Newton polytopes of the discriminants of $g(y)$ and $p(y)$ are equal 
and we conclude that $g(y)$ is quasi-ordinary. 
\end{proof}

\medskip
\noindent
For $k$-regular quasi-ordinary Weierstrass polynomials we can say more. 

\begin{Corollary}
\label{pack1}
Let $f(y)$ be a $k$-regular Kuo-Lu quasi-ordinary Weierstrass polynomial
and let $g(y) \in \bK[[\underline{x}]][y]$ be a Weierstrass polynomial which 
is an irreducible factor of $f^{(k)}(y)$. 
Then there exists $[B]\in E(f)$ with $B\in T_k(f)$  such that:
\begin{enumerate}
\item If $B'\in T(f)\cap [B]$ then $G_{B'}(z)$ and $F_{B'}(z)$ do not have common roots.
\item If $B'\in T_{k}(f)\backslash [B]$ then every root of $G_{B'}(z)$ is a root of ${\cal F}^{\oplus}_{B'}(z)$.
\item If $B'\in T(f)\setminus T_k(f)$  then $G_{B'}(z)$ is a non-zero constant polynomial.
\end{enumerate}
\end{Corollary}

\noindent \begin{proof} 
Take $B'\in T_k(f)$. Then by Lemma~\ref{derivatives} and the definition of $k$-regularity 
$G_{B'}(z)$ and ${\cal F}_{B'}^{\ominus}(z)$ do not have common roots. 
Hence for $B'\in T_k(f)$  it is enough to use Theorem~\ref{pack}. This proves
{\em 1.}   The second statement is  the first item of Theorem~\ref{pack}. 

\medskip
\noindent
Now let $B'\in T(f)\setminus T_k(f)$.
Consider the chain of bars  $B_0\perp_c B_1\perp \cdots \perp B_s=B'$  
of $T(f)$ such that $B_0\in T_k(f)$ and  $B_i\notin T_k(f)$ for $1\leq i \leq s$. 
By the $k$-regularity of $F_{B_0}(z)$, we get $G_{B_0}(c)\neq 0$. Since $g$ is compatible with $B_0$, after (\ref{eq:F,q1xxx}) of Lemma \ref{subst}, we have
\[g(\lambda_{B'}(\underline{x})+z\underline{x}^{h(B')})=g(\lambda_{B_0}(\underline{x})+c\underline{x}^{h(B_0)}+\cdots) = G_{B_0}(c)\underline{x}^{q(g,B_0)}+\cdots,\]
which shows that $g$ is also compatible with $B'$ and its $B'$-characteristic polynomial $G_{B'}(z)$ equals $G_{B_0}(c)$.\end{proof}

\section{Eggers factorizations of higher derivatives}
\label{section-Eggers factorization}
\medskip
\noindent 
Let $f$ be a quasi-ordinary Weierstrass polynomial. 
In this section we propose a factorization of the normalized derivative $f^{(k)}$ 
into factors associated with points of Eggers tree $E(f)$.  

\medskip

\begin{Definition}  Let $g,p\in \bK[[\underline{x}]][y]$ be Weierstrass polynomials. 
 The P-contact between $g$ and $p$ is 
\[
\cont_{\rm P}(g,p):=\frac{1}{\deg g \deg p}\Delta(\Res_y(g,p)).
\] 
\end{Definition}

\noindent The notion of P-contact has its counterpart in the theory of plane analytic curves: for $y$-regular plane branches it is related with the {\em logarithmic distance} studied by P\l oski in \cite{Ploski}, since in such case $\Delta(\Res_y(g,p))$ equals the Newton polygon of a monomial $x^m$, where $m$ is  the {\it intersection multiplicity} of the branches $g=0$ and $p=0$.

\medskip
\noindent If $g$ is compatible with a pseudo-ball $B$ then we put 
\[
\cont_{\rm P}(g,B):=\frac{1}{\deg g}\Delta(\underline x^{q(g,B)}).
\]

\begin{Proposition}\label{self-contact}
Let $B$ be a quasi-ordinary pseudo-ball of finite height and let $f$ be an irreducible quasi-ordinary Weierstrass polynomial 
compatible with $B$ such that $\Zer f \cap B\neq \emptyset$ or equivalently 
such that $F_B(z)$ has  positive degree. 
Then $\cont_{\rm P}(f,B)$ does not depend on $f$. 
\end{Proposition}

\noindent \begin{proof}
 Take any $f_1$, $f_2$ satisfying the assumptions of the proposition and let 
$\alpha_1 \in \Zer f_1 \cap B$, $\alpha_2 \in \Zer f_2 \cap B$. Choose a constant  $c\in \bK$ such that 
$(F_i)_B(c)\neq 0$,  for $i=1,2$
and let $\gamma =\lambda_B+c\underline x^{h(B)}$. 
Then $O(\gamma,\alpha_1)=O(\gamma,\alpha_2)= h(B)$ 
and for any  $\xi \in \Zer f_1\cup \Zer f_2$ we have $O(\xi,\gamma)\leq h(B)$. 

\medskip

\noindent Let $G$ be a finite subgroup of $\bL$-automorphisms of $\bM$ that acts transitively on the sets $\Zer f_1$ and $\Zer f_2$. By the orbit stabilizer theorem, for $i\in \{1,2\}$ we get 

\[ 
\frac{1}{|G|} \sum_{\sigma\in G} O(\gamma,\sigma(\alpha_i)) = 
     \frac{1}{\deg f_i} \sum_{\alpha\in\Zer f_i} O(\gamma,\alpha) = 
     \frac{1}{\deg f_i}q(f_i,B).
\]

\noindent By STI we have $O(\gamma,\sigma(\alpha_1))=O(\gamma,\sigma(\alpha_2) )$ for all $\sigma\in G$.
Thus  $\frac{1}{\deg f_1}q(f_1,B)= \frac{1}{\deg f_2}q(f_2,B)$.
\end{proof}
 
 \medskip
 \noindent  After Proposition~\ref{self-contact} we define 
 the \emph{self-contact} of a pseudo-ball $B$ of finite height as 
 \[
 \mbox{self-contact}(B):= \cont_{\rm P}(f,B),
 \] 
 
\noindent for any $f$ satisfying the assumptions of this proposition.
\medskip

\noindent By Lemma~\ref{LL:1} conjugate pseudo-balls have the same self-contact, hence the self-contact of $[B]$ is well-defined for any vertex $[B]$ of ${E}(f)$, where $B$ is of finite height. 
 
\medskip
\noindent 
In the set of  Newton polytopes we define the next partial order:  $\Delta_1 \succeq \Delta_2$ if and only if $\Delta_1 \subseteq \Delta_2$.  
Observe that $\Delta(\underline x^{{\bf q}_1})\succeq \Delta(\underline x^{{\bf q}_2})$ if and only if ${\bf q}_1 \geq {\bf q}_2$. Now we show how the self-contacts of $[B]\in E(f)$ determine the P-contacts between irreducible factors of $f$.

\begin{Proposition}\label{s-c1}
Let $f$ be a quasi-ordinary Weierstrass polynomial.
Then the self contacts of vertices of finite height increase along the branches of $E(f)$. Moreover for any different  irreducible factors $f_1$, $f_2$ of $f$
\begin{equation}
\cont_{\rm P}(f_1,f_2) = \max \{\mbox{\rm self-contact}([B]) \} \label{s-c},
\end{equation}
where the maximum is taken over all $[B]\in E(f)$
such that $\Zer f_i \cap B\neq \emptyset$ for $i=1,2$.
\end{Proposition}

\noindent \begin{proof}
Let $B$, $B'$ be pseudo-balls  of $T(f)$ of finite height such that $B'\subsetneq B$. Choose an irreducible factor $f_i$ of $f$ such that 
$\Zer f_i \cap B'\neq \emptyset$. By Lemma \ref{in-chain} we get $q(f_i,B)< q(f_i,B')$, hence $\mbox{self-contact}(B)\prec \mbox{self-contact}(B')$.

\medskip

\noindent Let $[B]\in E(f)$ be the maximum (with the order defined in $E(f)$) of the set of all vertices $[B']\in E(f)$ such that $\Zer f_i \cap B'\neq \emptyset$ for $i=1,2$. The pseudo-ball $B$ has the form $\gamma+(\gamma-\delta)\bK[[\underline{x}^{1/\bN}]]$, for some $\gamma \in \Zer f_1$ and $\delta \in \Zer f_2$ with maximal possible contact. By the choice of $\gamma$ and $	\delta$,   we have $O(\gamma,\delta')\leq h(B)$ for all $\delta'\in\Zer f_2\cap B$, consequently $(F_2)_{B}(\lc_B \gamma)\neq 0$. Then $f_2(\gamma)=(F_2)_B(\lc_B \gamma)\underline{x}^{q(f_2,B)}+\cdots$. 

\medskip

\noindent Applying the Galois action associated with the irreducible polynomial $f_{2}$ we get $\Delta({f_2(\gamma)})=\Delta({f_2(\gamma')})$, for any $\gamma,\gamma' \in \Zer f_{1}$. Hence by the definition of the self-contact and the identity
$\Delta(\Res_y(f_1,f_2))=\sum_{\gamma \in \Zer f_{1}}\Delta({f_2(\gamma)})$ we have
\begin{eqnarray*}
\mbox{self-contact}(B)&=&\cont_{\rm P}(f_2,B)=\frac{1}{\deg f_2}\Delta(\underline x^{q(f_2,B)})\\
&=& \frac{1}{\deg f_1 \deg f_2}\deg f_1\;\Delta({f_2(\gamma)})\\
&=&\frac{1}{\deg f_1 \deg f_2}\Delta(\Res_y(f_1,f_2))=\cont_{\rm P}(f_1,f_2).
\end{eqnarray*}
\end{proof}

\begin{Theorem}
\label{dec-red-qo}
Let $f \in \bK[[\underline{x}]][y]$  be a quasi-ordinary Weierstrass polynomial.
Then 
\[f^{(k)}=\prod_{[B]\in E(f)}p_{[B]},\] 
where  $p_{[B]}$ are Weierstrass polynomials such that
\begin{enumerate}
\item The $B$-characteristic polynomial of $p_{[B]}$ equals $F_B^{\ominus}$ 
up to multiplication by constants  and $\deg p_{[B]}=N(B)t_k(B)$. 
\item For every irreducible factor $g$ of $p_{[B]}$ and every irreducible factor $f_i$ of~$f$ we get
\begin{enumerate}
\item[(a)] $\cont_{\rm P}(g,B)=\mbox{\rm self-contact}(B).$

\item[(b)] If $\cont_{\rm P}(f_i,B)\prec \mbox{\rm self-contact}(B)$ then $\cont_{\rm P}(f_i,g)=\cont_{\rm P}(f_i,B)$.

\item[(c)] If $  \cont_{\rm P}(f_i,B)=\mbox{\rm self-contact}(B)$ then 
   $\cont_{\rm P}(f_i,g)\succeq \cont_{\rm P}(f_i,B)$.
\end{enumerate}

\item If $f$ is $k$-regular then the inequalities $\succeq$ in (c) become equalities.

\item For every irreducible factor $g$ of $p_{[B]}$ there is an irreducible factor $f_i$ of $f$ such that 
$\cont_{\rm P}(f_i,g) = \cont_{\rm P}(f_i,B)= \mbox{\rm self-contact}(B).$

\end{enumerate}
\end{Theorem}

\noindent \begin{proof}
We define $p_{[B]}$ as the  product of all irreducible factors of $f^{(k)}$ 
having the same $[B]$ in Theorem~\ref{pack} (by convention the product of an empty family is $1$).
After some monomial substitution  $\bar p_{[B]}$ has $N(B)t_k(B)$  roots 
and all of them are in $\bigcup_{B'\in[B]}\bar{B'}^0$. 
Consequently  $\deg p_{[B]}=N(B)t_k(B)$.\\

\noindent Now we will prove the second statement. Since $p_{[B]}$ has positive degree we may assume that $B\in T_{k}(f)$.
Let $f_i$ be an irreducible factor of $f$.
If $\cont_{\rm P}(f_i,B)\prec \mbox{ self-contact}(B)$ then by Proposition \ref {self-contact} $(F_i)_B(z)$ is a non-zero constant polynomial. Hence, for any $\bar\gamma\in \Zer \bar p_{[B]}$ we have $\ord\bar{f_i}(\bar{\gamma})=q(\bar f_i,\bar B)$, which proves {\em 2(b)}.
Suppose now that $\cont_{\rm P}(f_i,B)=\mbox{ self-contact}(B)$. For every root $\bar\gamma\in \Zer \bar p_{[B]}$ we
have $\ord\bar{f_i}(\bar{\gamma})\geq q(\bar f_i,\bar B)$ with equality in the $k$-regular case. Hence  if $g$ is an irreducible factor of $p_{[B]}$ then 
$\ord \Res_y(\bar f_i, \bar g )\geq (\deg g)\cdot q(\bar f_i,\bar B)$ with equality in the $k$-regular case. This gives  {\em 2(c)} and {\em 3}. \\

\medskip

\noindent If the polynomial $F_{B}(z)$ is as in \eqref{eq:F irr} then it is $k$-regular. In this case for any irreducible factor $f_i$ of $f$, with   $(F_i)_B(z)$ of  positive degree, the polynomials $(F_i)_B(z)$ and $G_B(z)$  do not have common factors. 

\medskip
\noindent If  $F_{B}(z)$ is not  as in \eqref{eq:F irr},  then by Lemma \ref{power} there is an irreducible factor $f_i$ of $f$ such that 
the polynomials $(F_i)_B(z)$ and $G_B(z)$  do not have common factors and $(F_i)_B(z)$ has positive degree. 
After any monomial substitution, we have 
$\ord \bar f_i(\bar\gamma)=q(\bar f_i,\bar B)$, 
for every $\bar\gamma \in \Zer \bar g$. 
This gives $\ord \Res(\bar f_i,\bar g)=\deg g \cdot q(\bar f_i,\bar B)$.
Since the monomial substitution was arbitrary, 
the fourth statement of the theorem holds true in all cases. 

\medskip
\noindent It rests to prove {\em 2(a)}. Choose $f_i$ as in the proof of the fourth statament. Then $\Delta(g(\alpha))=\Delta(\underline x^{q(g,B)})$ for any $\alpha \in B\cap \Zer f_i$.

\medskip
\noindent Applying the same argument as in the end of the proof of Proposition \ref{s-c1}, we get $\Delta(\Res_y(f_i,g))=\deg f_{i}\Delta(g(\alpha))=\deg f_i\cdot \Delta(\underline x^{q(g,B)})$. After the fourth statement and the definition of the P-contact: $\mbox{self-contact}(B)=\cont_{\rm P}(f_i,g)=\frac{1}{\deg g\deg f_i}\Delta(\Res_y(f_i,g))=
\frac{1}{\deg g}\Delta(\underline x^{q(g,B)})=\cont_{\rm P}(g,B).$
\end{proof}

\begin{Example}
We consider  the example in \cite[Section 10]{GB-GP}: let $f=f_{1,1}f_{1,2}f_{2,1}f_{2,2}$, 
where $f_{i,j}=(y^2-ix_1^{3}x_2^2)^2-jx_1^{5}x_2^{4}y$ 
are irreducible quasi-ordinary polynomials for $i,j\in\{1,2\}$. The Kuo-Lu and the  Eggers tree of $f$ are

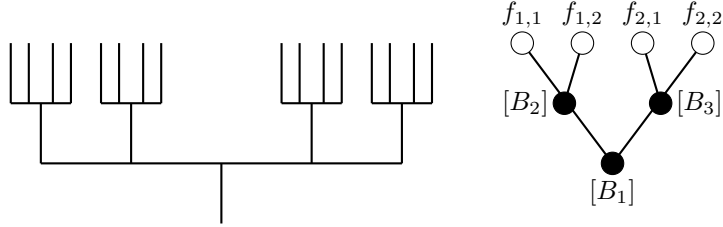
\begin{center}
\begin{tikzpicture}[scale=0.8]
\draw [-, thick](0,0) -- (0,1); 
\draw[-, thick](-3,1) -- (3,1); 
\draw[-, thick](-3,1) -- (-3,2); 
\draw[-, thick](3,1) -- (3,2); 
\draw[-, thick](-1.5,1) -- (-1.5,2); 
\draw[-, thick](1.5,1) -- (1.5,2); 
\draw[-, thick](-3.5,2) -- (-2.5,2); 
\draw[-, thick](-2,2) -- (-1,2);
\draw[-, thick](1,2) -- (2,2);
\draw[-, thick](2.5,2) -- (3.5,2);
\draw[-, thick](-3.5,2) -- (-3.5,3);
\draw[-, thick](-3.2,2) -- (-3.2,3);
\draw[-, thick](-2.8,2) -- (-2.8,3);
\draw[-, thick](-2.5,2) -- (-2.5,3);

\draw[-, thick](-1,2) -- (-1,3);
\draw[-, thick](-1.3,2) -- (-1.3,3);
\draw[-, thick](-1.7,2) -- (-1.7,3);
\draw[-, thick](-2,2) -- (-2,3);

\draw[-, thick](1,2) -- (1,3);
\draw[-, thick](1.3,2) -- (1.3,3);
\draw[-, thick](1.7,2) -- (1.7,3);
\draw[-, thick](2,2) -- (2,3);

\draw[-, thick](2.5,2) -- (2.5,3);
\draw[-, thick](2.8,2) -- (2.8,3);
\draw[-, thick](3.2,2) -- (3.2,3);
\draw[-, thick](3.5,2) -- (3.5,3);

 \begin{scope}[shift={(6.5,1)},scale=1]
\draw [thick](0,0) -- (-1.5,2); 
\draw[thick](0,0) -- (1.5,2); 
\draw [thick](0.8,1) -- (0.5,2); 
\draw[thick](-0.8,1) -- (-0.5,2) ; 
\draw (0,-0.1) node[below]{$[B_1]$};
\draw (-0.9,1) node[left]{$[B_2]$};
\draw (0.9,1) node[right]{$[B_3]$};
\draw (-1.5,2.1) node[above]{$f_{1,1}$};
\draw (-0.5,2.1) node[above]{$f_{1,2}$};
\draw (0.5,2.1) node[above]{$f_{2,1}$};
\draw (1.5,2.1) node[above]{$f_{2,2}$};
\node[draw,circle,inner sep=3pt,fill=black] at (0.8,1) {};
\node[draw,circle,inner sep=3pt,fill=black] at (-0.8,1) {};
\node[draw,circle,inner sep=3pt,fill=black] at (0,0) {};
\node[draw,circle,inner sep=3pt,fill=white] at (1.5,2) {};
\node[draw,circle,inner sep=3pt,fill=white] at (0.5,2) {};
\node[draw,circle,inner sep=3pt,fill=white] at (-0.5,2) {};
\node[draw,circle,inner sep=3pt,fill=white] at (-1.5,2) {};
\end{scope}
\end{tikzpicture}
\end{center}

\noindent The heights of the vertices of the Eggers tree are:
$h[B_1]=\left(\frac{3}{2},1\right)$, $h([B_2])=h([B_3])=\left(\frac{7}{4},\frac{3}{2}\right)$;  the self-contacts are $\hbox{\rm self-contact}([B_1])=\frac{1}{4}\Delta(\underline x^{(6,4)})$ and $\hbox{\rm self-contact}([B_2])=\hbox{\rm self-contact}([B_3])=\frac{1}{4}\Delta\left(\underline x^{(13,10)}\right)$.

\medskip

\noindent  For any $1\leq k \leq 16$, the degrees of polynomials $p_{[B_i]}$  are

\begin{center}
\begin{tabular}{|c|c|c|c|c|}
 \hline
&$\deg p_{[B_1]}$ & $\deg p_{[B_2]}$ & $\deg p_{[B_3]}$ \\ 
\hline \hline
$f^{(1)}$ & 3 &  6 & 6 \\
\hline
$f^{(2)}$ & 6 &  4 & 4 \\
\hline
$f^{(3)}$ & 9 &  2 & 2 \\
\hline 
$f^{(k)}$ & 16-$k$ &  0 & 0 \\
\hline 
\end{tabular}
\end{center}

\noindent The characteristic polynomials are $F_{B_1}(z)=(z^2-1)^4(z^2-2)^4$,
$F_{B_2}(z)=(4z^2-1)(4z^2-2)$ and 
$F_{B_3}(Z)=(8z^2-\sqrt{2})(8z^2-2\sqrt{2})$. We can verify that these polynomials are $k$-regular for any $k$. 

\noindent Theorem \ref {dec-red-qo} allows us to compute the {\rm P-contact} between the irreducible factors of $f$ and the irreducible factors of its higher order polars. For any $k$ and any irreducible factor $g$ of $p_{[B_{1}]}$, we have
$ \cont_{\rm P}(f_{i,j},g)=\hbox{\rm self-contact}([B_1])$. For any $k$ and any irreducible factor $g$ of $p_{[B_{2}]}$, we have
$ \cont_{\rm P}(f_{1,j},g)=\hbox{\rm self-contact}([B_2])$ and $ \cont_{\rm P}(f_{2,j},g)=\hbox{\rm self-contact}([B_1])$, for any $j=1,2$. We have the symmetric situation for the irreducible factors of $p_{[B_{3}]}$.

\end{Example}

\begin{Example}
The second polar of the quasi-ordinary polynomial $f$ from the Example \ref{ex:KL} (see page \pageref{Egg ex} for its Eggers tree) has only one Eggers factor $p_{[B]}=y$, with 
$\cont_{\rm P}(f_2,y)=\Delta (\underline x^{(5,2)})\succ \cont_{\rm P}(f_1,y)=\Delta (\underline x^{(3/2,1)})= \mbox{\rm self-contact}(B),$ hence in item 2. (c) of Theorem \ref{dec-red-qo} we have equality for $f_{1}$ and strict inequality for $f_{2}$.
\end{Example}

\noindent Now we study the examples of \cite{Casas}:

\begin{Example} (\cite[Example 5.1]{Casas})
Let $f=y^3+x^2y$.
The Eggers tree of $f$ has only one vertex $[B]$ of finite height, where $B=x\bK[[x^{1/\bN}]]$. The $B$-characteristic polynomial of $f$ is as in the previous example, so it is not $2$-regular.
We get $f^{(2)}=p_{[B]}=y$. If $f_{1}=y-x$, $f_{2}=y+x$ and $f_{3}=y$ then
$\emptyset=\cont_{\rm P}(f_3,y)\succeq \cont_{\rm P}(f_i,y)=\cont_{\rm P}(f_i,B)=\Delta(x)=\mbox{\rm self-contact}(B)$ for $i=1,2$. This illustrates the fourth statement of Theorem \ref{dec-red-qo}.
\end{Example}

\begin{Example} (\cite[Example 5.2]{Casas})
Let $f_a=y^4+ax^2y^2+x^2y+x^{10}$. We get $f_a=f_{a1}f_{a2}$, where  $f_{a1}$ is irreducible and the contact of any two  different roots of it is $\frac{2}{3}$, and $f_{a2}=0$ is a smooth curve tangent to $y=0$. The Eggers tree of $f_a$ is:

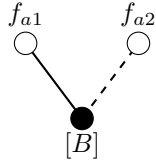
\begin{center}
\begin{tikzpicture}[scale=0.5]
\draw [thick](0,0) -- (-1.5,2); 
\draw[thick, dashed](0,0) -- (1.5,2);
\draw (0,-0.1) node[below]{$[B]$};
\draw (-1.5,2.2) node[above]{$f_{a1}$};
\draw (1.5,2.2) node[above]{$f_{a2}$};
\node[draw,circle,inner sep=3pt,fill=black] at (0,0) {};
\node[draw,circle,inner sep=3pt,fill=white] at (1.5,2) {};
\node[draw,circle,inner sep=3pt,fill=white] at (-1.5,2) {};
\end{tikzpicture}
\end{center}

\noindent The characteristic polynomial  $F_B(z)$ equals $z^4+z$. Hence 
$f_a$ is not $2$-regular. For any irreducible factor $g$ of $f^{(2)}_a$
we get  $\cont_{\rm P}(f_{a1},g)=\Delta(x^{2/3})=\mbox{ self-contact}(B)$. For $f_{a2}$ the P-contact depends on $a$:

\[
\cont_{\rm P}(f_{a2},g)=\left\{\begin{array}{ll}
\Delta(x) &  
\hbox{\rm for $a\neq 0$}\\
\Delta(x^{8})  & \hbox{\rm for $a=0$.}
\end{array}
\right.
\]
\end{Example}

\medskip

\subsection{Irreducible case}

\noindent Assume that $f(y)\in \bK[[\underline{x}]][y]$  is an irreducible
quasi-ordinary Weierstrass polynomial of degree $n>1$ and $\Zer f=\{\alpha_i\}_{i=1}^n$. 
By \cite{Lipman} the set $\{O(\alpha_i,\alpha_j)\,:\, i\neq j\}:=\{{\bf h}_1,\ldots, {\bf h}_s\}$ is well-ordered, so we may assume that  
${\bf h}_1\leq {\bf h}_2 \leq \cdots \leq {\bf h}_s$.  These values are the
finite heights of the bars of $T(f)$.
The sequence ${\bf h}_1,\ldots, {\bf h}_s$ is called the sequence of {\it characteristic exponents} of $f(y)$. Let $B_i$ be any bar in $T(f)$ of height ${\bf h}_i$. By \cite[Remark 2.7]{GP} the degree $n(B_i)$ of the field extension $\bL(\lambda_{B_i}(\underline{x})) \hookrightarrow \bL(\lambda_{B_i}(\underline{x}), \underline{x}^{{\bf h}_i})$ does not depend on the choice of $B_i$ and will be denoted by $n_i$. Put $e_i:=n_{i+1}\cdots n_s$ for $0\leq i \leq s$ (by convention the empty product is one).
Observe that $T(f)$ has a special structure: all bars of the same height are conjugate and there are $n_1\cdots n_{i-1}$ conjugate bars of height~${\bf h}_i$ (see \cite[Theorem 6.2]{IMRN}).

\medskip
\noindent  By~(\ref{jac_Nd}) we get
\begin{equation}
\label{deltaRes1}
\Delta((\Res_y (f^{(k)}, f-T))=\sum_{i=1}^s n_1\cdots n_{i-1}t_k(B_i)\left\{\Teissr{q(f,B_i)}{1}{8}{4}\right\},
\end{equation}

\noindent where  $B_i$ is any ball of $T(f)$ of height ${\bf h}_i$ and
\[
t_k(B_i)=\left\{
\begin{array}{ll}
(n_i-1)k &\hbox{ for $1\leq k\leq e_{i}$},\\
e_{i-1}-k     & \hbox{ for $e_{i}\leq k \leq e_{i-1}$},\\
0          & \hbox{ for $e_{i-1} \leq k < n  $}.
\end{array}
\right .
\]

\noindent Let $i_{k}\in \{1,\ldots, s\}$ be  such that $e_{i_k}\leq k< e_{i_k-1}$. Then
$t_k(B_i)$ is positive if and only if $1\leq i \leq i_{k}$.

\noindent The Newton polytope of (\ref{deltaRes1}) is polygonal (see Corollary \ref{coro:polygonal}) and has $i_{k}$ edges of different inclinations. After Theorem \ref{Th:decomp}
we decompose 
$\Res_y(f^{(k)}, f-T)=\prod_{i=1}^{i_{k}} R_i$, 
where  $\deg_T R_i=(n_1\cdots n_{i-1})t_k(B_i)$  and any $R_i$  
has an elementary Newton polytope of  inclination $q(f,B_i)$.

\medskip

\noindent Such a decomposition of the resultant can be also obtained from Eggers factorization of $f^{(k)}$. By Lemma \ref{power} the $B_{i}$-characteristic polynomial of $f$ has the form 
\begin{equation}
\label {eq:char-irred}
F_{B_{i}}(z)=\hbox{\rm constant}(z^{n_i}-c_{B_{i}})^{e_i},
\end{equation}
 
\noindent for some  $c_{B_i}\in \bK\setminus \{0\}$. The properties of such polynomials are described  in the following lemma, which was proved in  \cite[Lemma~5.3]{Forum} for complex polynomials but by Lefschetz Principle  it holds true for polynomials over any algebraically closed field of characteristic zero. 

\begin{Lemma}
\label{AL}
Let $\bK$ be an algebraically closed field of charactersitic zero. 
If $F(z)=(z^n-c)^{e}\in \bK[z]$ with $c\neq 0$ then for $1\leq k<\deg F(z)$ one has 
$\frac{d^k}{dz^k}F(z)=C z^{a}(z^n-c)^{b}\prod_{i=1}^{d}(z^n-c_i)$, where $C\neq 0$ and
\begin{enumerate}
\item[(1)] $0\leq a <n$ and $a+k\equiv 0 \pmod n$,
\item[(2)] $b=\max\{e-k,0\}$,
\item[(3)] $d=\min\{e,k\}-\lceil \frac{k}{n} \rceil$,  where $\lceil x \rceil$ denotes the smallest integer  bigger than or equal to $x$,
\item [(4)] $c_i\neq c_j$ for $1\leq i<j\leq d$ and $0\neq c_i\neq c$ for $1\leq i\leq d$.
\end{enumerate}
\end{Lemma}

\begin{Corollary}
\label{irr-kregular}
Every  irreducible quasi-ordinary Weierstrass polynomial is Kuo-Lu 
$k$-regular for any positive  integer $k$. 
\end{Corollary}

\begin{Theorem}\label{Merle}
Let $f(y)\in \bK[[\underline{x}]][y]$  be an irreducible quasi-ordinary Weierstrass polynomial 
of degree $n>1$ and characteristic exponents ${\bf h}_1, \ldots, {\bf h}_s$. Then
\begin{equation}
\label{dM}
f^{(k)}(y)= \prod_{i=1}^{i_{k}}p_i,
\end{equation}
where  
\begin{enumerate}
\item $p_i$ is a Weierstrass polynomial in $\bK[[\underline{x}]][y]$ of degree $n_1\cdots n_{i-1}t_k(B_i)$.
\item Any irreducible factor $g$ of $p_i$ verifies
\[
\cont_{\rm P}(g,f)=\mbox{self-contact}(B_i).
\]
\item The $B_i$-characteristic polynomial of $p_i$ is 
$(P_i)_{B_i}=const F_{B_i}^{\ominus}.$ 
\end{enumerate}
\end{Theorem}

\noindent \begin{proof}
\noindent The theorem follows from Corollary \ref{irr-kregular} and the first, second and third part of  Theorem~\ref{dec-red-qo}.
\end{proof}

\medskip

\begin{Proposition}
\label{ppppp}
Let $f(y)\in \bK[[\underline{x}]][y]$  be an irreducible quasi-ordinary Weierstrass polynomial with characteristic exponents ${\bf h}_1, \ldots, {\bf h}_s$. 
Let $a$, $d$ be integers such that $0\leq a <n_i$, $a+k\equiv 0 \pmod {n_i}$ and 
$d=\min\{e_i,k\}-\lceil \frac{k}{n_i} \rceil$.
Then every $p_i$ of (\ref{dM}) admits a factorization of the form  
$p_i=p_{i0}p_{i1}\cdots p_{id},$ where

\begin{enumerate}

\item 
the corresponding $B_i$-characteristic polynomials are
$P_{i0}(z)=\mbox{const}\cdot z^a$, 
$P_{ij}(z)=\mbox{const}\cdot (z^{n_i}-c_j)$ 
with $c_j\neq c_l$ for $1\leq j<l\leq d$ and $c_j\neq0$.

\item $p_{i0}$  is a Weierstrass polynomial of degree $a\cdot n_1\cdots n_{i-1}$ not necessarily 
quasi-ordinary. 
\item  Every $p_{ij}$ for $1\leq j \leq d$ is a quasi-ordinary irreducible Weierstrass polynomial  
of degree $n_1\cdots n_i$ and characteristic exponents ${\bf h}_1, \ldots, {\bf h}_i$.
\end{enumerate}
\end{Proposition}

\noindent \begin{proof}
\noindent  After \eqref {eq:char-irred} $F_{B_i}(z)$ has the form $a(z^{n_i}-c)^{e_i}$ for some nonzero 
$a$ and $c$.

\noindent By the first part of Theorem \ref{dec-red-qo} and Lemma \ref{AL} the polynomial $P_{i,B_i}(z)=\hbox{\rm const}\cdot z^{a}\prod_{j=1}^{d}(z^{n_{i}}-c_j)$. This polynomial is the product of the $B_i$-characteristic polynomials of the irreducible factors of $p_{i}$.  
From the second part of Theorem~\ref{pack}, we know that $p_{i}$ has $d$ irreducible factors $\{p_{ij}\}_{j=1}^{d}$ such that $P_{ij}(z)=\mbox{const}\cdot (z^{n_i}-c_j)$. If $p_{i}$ has other irreducible factors, then $p_{i0}$ is their product.
It also follows from  Theorem~\ref{pack} that $p_{ij}$ are quasi-ordinary for $1\leq j\leq d$. \\

\noindent By a similar argument as in the first part of the proof of Theorem \ref{dec-red-qo} we get $\deg p_{ij}=N(B_{i})\deg P_{i,jB_i}(z)$. Since $N(B_{i})=n_{1}\cdots n_{i-1}$, we obtain the statements about the degrees of $p_{ij}$.\\

\noindent Fix $p_{ij}$ for $j\in\{1,\ldots,d\}$. The pseudo-ball $B_{i}$ has $n_{1}\cdots n_{i-1}$ conjugate pseudo-balls. Each of these pseudo-balls contains $n_{i}$ roots of $p_{ij}$. Since the roots of $P_{i,j}(z)$ are simple, 
any two roots of $p_{ij}$ belonging to the same pseudo-ball have different leading coefficients with respect to $B_i$, so their contact equals ${\bf h}_{i}$. Now, if we consider two roots of $p_{ij}$ belonging to different conjugate pseudo-balls,  then their contact depends only on these two pseudo-balls, hence it is equal to ${\bf h}_{l}$ for some $l\in\{1,\ldots, i-1\}$. We conclude that  the characteristic exponents of $p_{ij}$ are ${\bf h}_1, \ldots, {\bf h}_i$.
\end{proof}

\medskip

\noindent In Proposition \ref{ppppp} the integer $a$ can be $0$,  in such a case $p_{i0}=1$. If $a=1$ then $p_{i0}$ is quasi-ordinary with characteristic exponents ${\bf h}_1, \ldots, {\bf h}_{i-1}$. Moreover $d$ can be zero and in such a case $p_i=p_{i0}$.

\section{Eggers decomposition for power series}
\label{section-Eggers series}
\noindent 
In this section we deal with power series in variables $\underline{x}$ and $y$. 
A power series will be called {\it quasi-ordinary} if it is a product of a unity and a quasi-ordinary Weierstrass polynomial. 
We outline how to generalize the results of previous sections to quasi-ordinary power series. 
For that  we need the next generalization of Lemma~\ref{derivatives}:

\begin{Lemma}\label{derivatives-series}
Let $f=u f^*$ and 
$\frac{\partial ^k}{\partial y^k}f=w g^*,$ where 
$u$,  $w\in\bK[[\underline{x},y]]$ are unities,  
$f^*$,  $g^*\in \bK[[\underline{x}]][y]$ are Weierstrass polynomials 
and $1\leq k \leq n=\deg f^*$.
Assume that $f^{*}$ is compatible with a pseudo-ball $B$.
 Then $g^{*}$ is compatible with $B$ and 
 $G^*_B(z)=\frac{(n-k)!}{n!}\frac{d^k}{dz^k}F_B^*(z)$.
\end{Lemma}

\noindent \begin{proof} 
Substituting $\underline{x}=0$ we get 
$f(0,y)=u(0,0)y^n+\cdots$. 
Hence
$\frac{\partial ^k f}{\partial y^k}(0,y)=\frac{n!}{(n-k)!}u(0,0)y^{n-k}+\cdots$. 
On the other hand $\frac{\partial ^k f}{\partial y^k}(0,y)=w(0,y)g^{*}(0,y)$ which implies that
\begin{equation}
\label{eqw}
w(0,0)=\frac{n!}{(n-k)!}u(0,0).
\end{equation}

\medskip

\noindent By the assumption of compatibility of $f^{*}$ we have
\[
f^*(\underline{x},\lambda_B(\underline{x}) 
+ z\underline{x}^{h(B)})=F^*_B(z)\underline{x}^{q(f^*,B)}+\cdots .
\]
Hence 
$f_1(\underline{x},z):=\underline{x}^{-q(f^*,B)}f(\underline{x},\lambda_B(\underline{x}) + z\underline{x}^{h(B)})$ is a fractional 
power series such that 
\begin{equation}\label{eq:11.1}
f_1(0,z)=u(0,0)F^*_B(z) .
\end{equation}
By the chain rule of differentiation 
\begin{equation}\label{eq:11.2}
\frac{\partial ^k f_1}{\partial z^k}(\underline{x},z)=
    \frac{\partial ^k f}{\partial y^k} (\underline{x},\lambda_B(\underline{x})+z\underline{x}^{h(B)}) 
    \cdot\underline{x}^{k h(B)-q(f^*,B)}.
\end{equation}
Differentiating (\ref{eq:11.1}) yields $\frac{\partial ^k f_1}{\partial z^k}(0,z)=u(0,0)\frac{d^k}{dz^k}F_B^*(z)$. Thus
\begin{equation}\label{eq:11.3}
\frac{\partial ^k f_1}{\partial z^k}(\underline{x},z)=u(0,0)\frac{d^k}{dz^k}F_B^*(z)+\hbox{\rm terms of positive degree in $\underline{x}$}.
\end{equation}
Comparing (\ref{eq:11.2}) and (\ref{eq:11.3}) we get
\[
\frac{\partial ^k f}{\partial y^k} (\underline{x},\lambda_B(\underline{x})+z\underline{x}^{h(B)}) =
u(0,0)\frac{d^k}{dz^k}F_B^*(z)\cdot\underline{x}^{q(f^*,B)-k h(B)} + \cdots
\]
By the definition of $g^{*}$, the left hand side of the above equality can be written as
\[
w(0,0)\,g^*(\underline{x},\lambda_B(\underline{x})+z\underline{x}^{h(B)})+\cdots, 
\]
which gives, after (\ref{eqw})
\[
\frac{n!}{(n-k)!}u(0,0) g^*(\underline{x},\lambda_B(\underline{x})+z\underline{x}^{h(B)}) =
u(0,0)\frac{d^k}{dz^k}F_B^*(z)\cdot\underline{x}^{q(f^*,B)-k h(B)} + \cdots
\]

\noindent and finishes the proof. 
\end{proof}

\medskip
\noindent 
Theorem~\ref{higher-Kuo-Lu},
Corollary~\ref{Kuo-Lu},
Theorem~\ref{higher-discriminants},
Theorem~\ref{pack},
Corollary~\ref{pack1},
Theorem~\ref{dec-red-qo},
Theorem~ \ref{Merle} and Proposition \ref{ppppp}, where $f^{(k)}$ stands for the Weierstrass polynomial 
of $k$th derivative, remain true for quasi-ordinary power series. 
For the proofs it is enough to replace the power series by their Weierstrass polynomials and 
use Lemma~\ref{derivatives-series} instead of Lemma~\ref{derivatives} when required.

\medskip
\noindent
{\small Evelia Rosa Garc\'{\i}a Barroso\\
Departamento de Matem\'aticas, Estad\'{\i}stica e I.O.\\
Secci\'on de Matem\'aticas, Universidad de La Laguna\\
Apartado de Correos 456\\
38200 La Laguna, Tenerife, Espa\~na\\
e-mail: ergarcia@ull.es}

\medskip

\noindent {\small   Janusz Gwo\'zdziewicz\\
Institute of Mathematics\\
Pedagogical University of Krak\'ow\\
Podchor\c a{\accent95 z}ych 2\\
PL-30-084 Cracow, Poland\\
e-mail: janusz.gwozdziewicz@up.krakow.pl}

\end{document}